\documentclass[12pt,twoside,letterpaper]{article}

\usepackage{amsmath}
\usepackage{amssymb}
\usepackage{amsfonts}
\usepackage{amsthm}
\usepackage{verbatim}
\usepackage{subfig}
\usepackage{graphicx}

\def\cM{{\mathcal {M}}}

\def\p{\partial}

\def\R{{ \mathbb{R}}}
\def\N{{\mathbb{N}}}
\def\Z{{\mathbb{Z}}}

\def\C{{\mathbb{C}}}

\def\cM{{\cal M}}
\def\di{\diamond}
\def\di{\diamond}

\def\br{{\bf r}}

\newtheorem{theorem}{Theorem}[section]
\newtheorem{lemma}[theorem]{Lemma}
\newtheorem{proposition}[theorem]{Proposition}

\newtheorem{corollary}[theorem]{Corollary}

\theoremstyle{definition}
\newtheorem{definition}[theorem]{Definition}

\newtheorem{assumption}[theorem]{Assumption}

\newtheorem{remark}[theorem]{Remark}

\newtheorem*{remark*}{Remark}

\numberwithin{figure}{section} \numberwithin{equation}{section}

\title{Eigenvectors and eigenfunctionals of homogeneous order-preserving maps}

\author{Horst R. Thieme\thanks{({\tt hthieme@asu.edu})}
\\
 School of Mathematical and Statistical Sciences
 \\
Arizona State University, Tempe, AZ 85287-1804, USA
}

\begin{document}
\date{Feb 15, 2013}
\smallskip

\maketitle

\begin{abstract}
This paper considers  homogeneous order preserving
continuous maps on the normal cone of an
ordered normed vector space. It is shown that
certain operators of that kind which are not necessarily compact
themselves but have a compact power have
a positive eigenvector that is associated with
the cone spectral radius. We also derive conditions
for the existence of homogeneous order preserving
eigenfunctionals. Our results are illustrated in
a model for spatially distributed two-sex populations.
\end{abstract}


\pagestyle{myheadings}\markboth{\sc  H.R. Thieme} {\sl
Eigenvectors of homogeneous maps}


\section{Introduction}

For a linear bounded operator map $B$ on a complex Banach space,
the spectral radius of $B$ is defined as
\begin{equation}
\br(B) = \sup \{|\lambda |; \lambda \in \sigma (B)\},
\end{equation}
where $\sigma (B)$ is the spectrum of $B$,
\begin{equation}
\sigma(B) = \C \setminus \rho (B),
\end{equation}
and $\rho(B)$ the resolvent set of $B$, i.e., the set of those
$\lambda \in \C$ for which $\lambda - B$ has a bounded everywhere
defined inverse. The following alternative formula holds,
\[
\br(B) = \inf_{n\in \N} \|B^n\|^{1/n} = \lim_{n\to \infty} \|B^n\|^{1/n},
\]
which is also meaningful in a real Banach space. If $B$ is a compact linear map
on a complex Banach space
and $\br(B) >0$, then there exists some $\lambda \in \sigma(B)$ and $v \in X$ such
that  $|\lambda| = \br (B)$ and $Bv = \lambda v\ne0$. Such an $\lambda $ is called
an eigenvalue of $B$. This raises the question whether $\br(B)$  could be an eigenvalue itself. There is a positive answer, if $B$ is a positive operator
and satisfies some generalized compactness assumption.

\subsection{Positivity}

For models in the biological, social, or economic sciences, there
is a natural  interest in solutions that are positive in an
appropriate sense.

 A closed subset $X_+$ of a normed real  vector space $X$ is called a
 {\em  wedge} if

\begin{itemize}

\item[(i)] $X_+$ is convex,

\item[(ii)] $\alpha x \in X_+$ whenever $x \in X_+$ and $\alpha \in \R_+$.

\end{itemize}

A  wedge is called a {\em cone} if

\begin{itemize}
\item[(iii)] $X_+ \cap (-X_+) = \{0\}$.
\end{itemize}

Nonzero points in a cone or wedge are called {\em positive}.

A wedge is called {\em solid} if it contains interior points.

A wedge is called {\em reproducing} (also called {\em generating}) if
\begin{equation}
\label{eq:reproducing}
X = X_+ - X_+ ,
\end{equation}
and {\em total} if $X$ is the closure of $X_+- X_+$.

A  cone $X_+$ is called {\em normal}, if
there exists some $\delta > 0$ such that
\begin{equation}
\label{eq:normal}
 \|x +z \| \ge \delta  \hbox{ whenever } x \in X_+, z \in X_+, \|x\|=1 = \|z\|.
\end{equation}
Equivalent conditions for a cone to be normal are given in Theorem
\ref{re:cone-normal}.
In function spaces, typical cones are  formed by the nonnegative functions.

A map $B$ on $X$ with $B(X_+) \subseteq X_+$ is called a {\em positive map}.

If $X_+$ is a cone in $X$, we introduce a partial order on $X$ by
$x \le y $ if $y-x \in X_+$ for $x,y \in X$.

\begin{definition}
\label{def:order-pres}
Let $X$ and $Z$ be ordered vector space with cones $X_+$ and $Z_+$ and  $ U \subseteq X$. A map $B: U \to Z$ is called
{\em order preserving}  (or monotone or increasing) if $B x \le By$ whenever $x,y \in U$ and $x \le y$.
\end{definition}

Positive linear maps are order-preserving.
They have the remarkable property that their spectral radius is a spectral
value \cite{Bon58} \cite[App.2.2]{Sch} if $X$ is a Banach space and
$X_+$ a normal reproducing cone.

The celebrated Krein-Rutman theorem \cite{KrRu},
which generalizes parts of the Perron-Frobenius theorem to
infinite dimensions,
establishes that a compact positive linear map
$B$ with $\br(B) > 0$ on an ordered Banach space $X$ with total  cone $X_+$ has an eigenvector
$v \in X_+$, $v \ne 0$, such that $Bv = \br(B)v$ and a
positive bounded linear eigenfunctional
$v^*: X \to \R$, $v \ne 0$, such that $x^* \circ B = \br(B) x^*$.

This theorem has been generalized into various directions by Bonsall \cite{Bon55} and Birkhoff \cite{Bir}, Nussbaum \cite{Nus81, Nus}, and Eveson and Nussbaum \cite{EvNu}
 (see these papers
for additional references).

Can the Krein-Rutman theorem be extended to certain nonlinear maps on cones?

\subsection{The space of bounded homogeneous maps}

In the following, $X$, $Y$ and $Z$ are ordered normed vector spaces
with cones $X_+$, $Y_+$ and $Z_+$ respectively,

\begin{definition}
 $B:X_+ \to Y$ is called {\em  (positively) homogeneous (of degree one)},
 if  $B(\alpha x) = \alpha B x$
for all $\alpha \in \R_+$, $x \in X_+$.
\end{definition}

Since we do not consider maps that are homogeneous in other ways,
we will simply call them homogeneous maps.
If follows from the definition that
\[
B0 = 0.
\]
Homogeneous maps are not Frechet differentiable at 0 unless $B(x+y) = Bx + By$
for all $x,y \in X_+$. For the following holds.

\begin{proposition}
Let $B: X_+ \to Y$ be homogeneous. Then the directional derivatives of $B$
exist at 0 in all directions of the cone and
\[
\p B (0,x) = \lim_{t \to 0+}\frac{ B(t x) - B(0)}{t} = B(x), \qquad x \in X_+.
\]
\end{proposition}

There are good reasons to consider homogeneous maps. Here is a mathematical one.

\begin{theorem} Let $F: X_+ \to Y$ and $u \in X$. Assume that
the directional derivatives of $F $ at $u$ exist in all directions of
the cone. Then the map $B: X_+ \to X_+$, $B = \p F(u, \cdot)$,
\[
B (x) = \p F(u,x) = \lim_{t \to 0+} \frac{F(u+ t x) - F(u)}{t}, \qquad x \in X_+,
\]
is homogeneous.
\end{theorem}

\begin{proof}
Let $\alpha \in \R_+$. Obviously, if $\alpha=0$, $B(\alpha x) =0 = \alpha B(x)$.
So we assume $\alpha \in (0,\infty)$. Then
\[
\frac{F (u + t [\alpha x]) - F(u)}{t}
=
\alpha \frac{F(u +  [t\alpha] x) - F(u)}{t\alpha } .
\]
As $t \to 0$, also $\alpha t \to 0$ and so the directional derivative
in direction $\alpha x$ exists and
\[
\p F (u , \alpha x) = \alpha F(u,x). \qedhere
\]
\end{proof}

Another good reason are mathematical population models that take into
account that, for many species, reproduction involves a mating process
between two sexes.
The map involved therein is not only homogeneous but also
order preserving (Section \ref{sec:two-sex}).
The spectral radius of a positive linear map has gained
considerable notoriety because of its relation to
the basic reproduction number of population
models which have a highly dimensional structure but implicitly
assume a one to one sex ratio \cite{CuZh, DiHeMe,Thi09, DrWa}. A spectral radius
for homogeneous order-preserving maps should play
a similar role as an extinction versus persistence  threshold parameter for structured populations with two sexes.

For a homogeneous map $B:X_+ \to Y$, we define
\begin{equation}
\label{eq:operator-norm}
\|B\|_+ = \sup \{ \|Bx \|; x \in X_+, \|x\| \le 1 \}
\end{equation}
and call $B$ {\em bounded} if this supremum is a real number.
Since $B$ is  homogeneous,
\begin{equation}
\label{eq:operator-norm-est}
\|Bx \| \le \|B\|_+\, \|x\| , \qquad x \in X_+.
\end{equation}

Let $H(X_+,Y)$ denote the set of bounded homogeneous maps $B:X_+ \to Y$
and $H(X_+,Y_+)$ denote the set of bounded homogeneous maps $B:X_+ \to Y_+$
  and {\rm HM}$(X_+, Y_+)$ the set of those maps in $H(X_+,Y_+)$ that are also order-preserving.

 $H(X_+, Y)$ is a real vector space
and $\|\cdot \|_+$ is a norm on $H(X,Y_+)$;
$H(X_+,Y_+)$ and {HM}$(X_+, Y_+)$ are cones in $H(X_+,Y)$.
We write $H(X_+) = H(X_+, X_+)$ and HM$(X_+) = $HM$(X_+,X_+)$.

It follows for  $B \in H(X_+, Y_+)$ and
 $C\in H(Y_+, Z_+)$ that
$C B \in H(X_+, Z_+)$ and
\[
\|  C B\|_+ \le \|C\|_+\, \|B\|_+.
\]

\subsection{Cone and orbital spectral radius for bounded homogeneous maps}

Let $B \in H(X_+)$ and define $\phi: \Z_+ \to \R$ by $\phi (n) = \ln \|B^n \|_+$.
Then $\phi (m+n) \le \phi(m) + \phi (n)$ for all $m,n \in \Z_+$,
and a well-known result implies the following formula for the
{\em cone spectral radius}
\begin{equation}
\label{eq:cone-spec-rad}
\br_+(B):= \inf_{n \in \N} \|B^n\|_+^{1/n} = \lim_{n \to \infty} \|B^n\|_+^{1/n}.
\end{equation}

Mallet-Paret and Nussbaum \cite{MPNu02, MPNu} suggest an alternative definition of
a spectral radius for  homogeneous bounded maps
$B : X_+ \to X_+$. First, define an asymptotic least upper bound
for the geometric growth factor of $B$-orbits,
\begin{equation}
\label{eq:growth-factor}
\gamma_B(x) := \limsup_{n\to \infty} \|B^n x \|^{1/n}, \qquad x \in X_+,
\end{equation}
and then
\begin{equation}
\label{eq:spec-rad-orb}
\br_o(B) = \sup_{x \in X_+} \gamma_B (x).
\end{equation}
The number $\br_+(B)$ has been called {\em partial spectral radius}
by Bonsall \cite{Bon58},
$X_+$ spectral radius by Schaefer \cite{Sch59,Sch}, and
{\em cone spectral radius} by Nussbaum \cite{Nus}.
Mallet-Paret and Nussbaum \cite{MPNu02, MPNu} call $\br_+(B)$
the {\em Bonsall cone spectral radius} and $\br_o(B)$  the cone spectral radius.
For $x \in X_+$, the number $\gamma_B(x)$ has been called {\em local spectral radius} of $B$
at $x $ by F\"orster and Nagy \cite{FoNa}.

We will
follow Nussbaum's older terminology which  shares the spirit  with Schaefer's
\cite{Sch59} term
$X_+$ {\em spectral radius} and
stick with {\em cone spectral radius} for $\br_+(B)$
and call $\br_o(B)$ the {\em orbital spectral radius} of $B$.

One readily checks that
\begin{equation}
\br_+(\alpha B) = \alpha \br_+ (B), \quad \alpha \in \R_+,
\qquad
\br_+ (B^m ) = (\br_+ (B))^m , \qquad m \in \N,
\end{equation}
and that the same properties hold for $\br_o(B)$.

The cone spectral
radius and the orbital spectral radius are meaningful if $B$ is just positively homogeneous and bounded,
but as in \cite{MPNu02, MPNu} we will be mainly interested in the case
that $B$ is also order-preserving and continuous.

Though the two concepts  coincide for many practical purposes,
they are both useful.

\begin{theorem}
\label{re:spec-rad-alt-char}
Let $X$ be an ordered normed vector with  cone $X_+$
and $B: X_+ \to X_+$ be continuous, homogeneous and
order preserving.

Then $\br_+(B)\ge \br_o(B) \ge \gamma_B(x)$, $x \in X_+$.

Further $\br_o(B) = \br_+(B)$ if one of the following
hold:

\begin{itemize}

\item[(i)] $X_+$ is complete and normal.

\item[(ii)] A power of $B$ is compact.

\item[(iii)] $X_+$ is normal and a power of $B$ is uniformly order bounded.
\end{itemize}
\end{theorem}

For the concepts and the proof of (iii) see Section \ref{sec:root-power}.
The other two conditions for equality have been proved in \cite{MPNu02},
Theorem 2.2 and Theorem 2.3 (the overall assumption
of \cite{MPNu02} that $X$ is a Banach space is not used
in the proofs.)

For a homogeneous bounded map, the definitions of the cone and orbital spectral radius have
somewhat lost the connection to the spectrum of the map  which is
difficult to generalize. But
at least one would like to know whether $\br_+(B)$ or $\br_o(B)$
are eigenvalues of $B$, i.e., is there some $x \in X_+$, $x\ne 0$,
such that $Bx = \br_+(B) x$ ($Bx  =\br_o(B)x$).

Already the seminal work by Krein and Rutman \cite{KrRu}
(see also \cite[Thm.4.3]{Bon62}) establishes the existence
of some $x \in X_+$, $x\ne 0$, and some $\lambda > 0$
such that $Bx = \lambda x$. $B$ is assumed to be continuous, compact, and dominant
with the latter meaning that there is some $c > 0$ and $u \in X_+$,
$u\ne 0$, such that $Bu \ge c u$. The eigenvalue $\lambda $ then
turns out to satisfy $\lambda \ge c$ but no connection is made
between $\lambda$ and some sort of spectral radius of $B$.
The homogeneity of $B$ can be dropped if $B$ is assumed to
be strictly positive in an appropriate sense and the cone is
normal \cite{Sch55} (see also \cite[Thm.4.2]{Bon62}).

Bohl \cite{Boh66} \cite[III.2]{Boh}  also proves the existence
 of a positive eigenvector. Under his stronger assumptions it is
clear that the eigenvalue $\lambda$ is the cone and orbital spectral radius though
he does not introduce these concepts either. Bohl assumes that the cone is solid,  some power of $B$ is compact and that $B$ is order preserving in some strict sense (actually
it is enough that $B$ commutes with such a map). Bohl's proof is constructive
as it provides the convergence of the (ratio) power method (von Mises
procedure \cite{vMi}) to the eigenvector $x$ and the eigenvalue $\lambda$,
and it also establishes convergence from below and above to the
spectral radius by what are sometimes
called Collatz-Wielandt numbers \cite{Col1,FoNa,Wie}.

A little  later, Nussbaum \cite{Nus85, Nus87, Nus88} and, more recently,
Mallet-Paret and Nussbaum \cite{MPNu02, MPNu} have found  eigenvectors of  homogeneous order preserving operators assuming
different compactness and monotonicity properties than Krein/Rutman
and Bohl and not requiring a dominance property of the map or
the solidity of the cone. Mallet-Paret and Nussbaum \cite{MPNu02, MPNu}
also establish that the eigenvalue can be chosen as the orbital
spectral radius.

Related to problems with beer barrel scent \cite{MPNu-beer, Tro},
it  appears that there are some
questions left when $B$ is not compact itself
and does not satisfy monotonicity properties strong enough
that $B$ becomes a strict contraction with respect
to Hilbert's projective metric \cite{Nus87, Nus88, Nus89}.  Our approach
is as old as the earliest literature in this direction
\cite[Sec.2.2]{Kra} \cite[Sec.3]{Sch55}, namely
to approximate $B$ by operators $B_n$ that have strong monotonicity
properties. The crux of this approach is to what degree  $B_n$ inherits
generalized compactness properties from $B$.

In the spirit of the Krein-Rutman theorem which  involves
the existence of eigenfunctionals as well,
we also investigate whether there is a homogeneous, order preserving,
bounded eigenfunctional $\phi: X_+ \to \R_+$ such that $\phi \circ B = \br_+(B) \phi\ne 0$.


\section{More on cones}

While the Krein-Rutman theorem and its generalizations
to homogeneous maps
do not require the cone to be normal,
we found that we can make only little further progress  if normality
is not assumed.

\subsection{Normal cones}

The following result is well-known \cite[Sec.1.2]{Kra}.

\begin{theorem}
\label{re:cone-normal}
Let $X$ be an ordered normed vector space with cone $X_+$.
Then the following three properties are equivalent:

\begin{itemize}

\item[(i)] $X_+$ is normal:
There exists some $\delta > 0$ such that $\|x +z \|\ge \delta$
whenever $x \in X_+$, $z \in X_+$ and $\| x \| =1 = \|z\|$.

\item[(ii)] The norm is semi-monotonic: There exists some $M \ge 0$
such that $\|x\| \le M \|x + z\|$ for all $x,z \in X_+$.

\item[(iii)] There exists some $\tilde M \ge 0$ such that
$\|x\| \le \tilde M \|y\|$ whenever $x \in X$, $y \in X_+$,
and $- y \le x \le y$.

\end{itemize}
\end{theorem}

\begin{remark} If $X_+$ were just a wedge, property (iii) would be
rewritten as

\begin{quote}
There exists some $\tilde M \ge 0$ such that
$\|x\| \le \tilde M \|y\|$ whenever $x \in X$, $y \in X_+$,
and $ y+x \in X_+$, $y -x \in X_+$.
\end{quote}

Notice that this property implies that $X_+$ is cone:
If $x \in X_+$ and $-x \in X_+$, then $0+x \in X_+$
and $0 -x \in X_+$ and (iii) implies $\|x\| \le \tilde M \|0\|=0$.

\end{remark}

\begin{proposition}[cf. {\cite[(4.2)]{KrLiSo}}]
\label{re:normal-mon-funct}
Let $X$ be an ordered normed vector space with  cone
$X_+$. We define $\psi: X \to \R_+$ by
\begin{equation}
\psi(x) = \inf \{ \|y \|; y \in X, y \ge x\}, \qquad x \in X.
\end{equation}
Then $\psi$ is homogenous, order preserving, and subadditive
($\psi(x+y) \le \psi(x)+ \psi(y)$, $x,y \in X$),
\[
\begin{split}
|\psi(x) - \psi(y)|\le & \;\|x-y\|, \qquad x,y \in X,
\\
\psi(x) = &\;0, \qquad x \in -X_+.
\end{split}
\]
Moreover, $\psi $ is strictly positive: $\psi(x) > 0$ for all $x \in X_+$, $x \ne 0$.

If $X_+$ is normal,  there exists some $\delta > 0$ such that $\delta \|x\| \le \psi(x)
\le \|x\|$ for all $x \in X_+$.
\end{proposition}

\begin{proof}
The functional $\psi$ inherits positive homogeneity  from
the norm. That $\psi$ is order-preserving is immediate from
the definition. For all $x \in X$,  $x \le x$ and so $\|x\| \ge \psi (x)$.

To show the subadditivity, let $u,v,x,y \in X$ and $u \ge x $ and $v \ge y$.
Then $u+v \ge x +y$ and $ \psi(x+y ) \le \|u +v \| \le \|u\| +\|v\|$.
Fix $u$ for a moment. Then $ \psi(x+y) - \|u\| \le \|v\|$ for all $v \ge y$.
Thus $\psi(x+y) - \|u\| \le \psi(y)$. Hence $\psi(x+y) - \psi(y) \le \|u\|$
for all $u \ge x$
and $\psi(x+y) - \psi(y) \le \psi(x)$.

Since $\psi$ is subadditive,
\[
| \psi (x) - \psi(y)| \le \psi(x-y) \le  \|x-y\|.
\]
To show that $\psi$ is strictly positive, let $x \in X_+$ and $\psi(x) =0$.
By definition, there exists a sequence $(y_n)$ in $X$ with $\|y_n\| \to 0$
and $y_n \ge x $ for all $n \in \N$. Then $y_n -x  \in X_+$. Since $X_+$
is closed, $-x = \lim_{n\to \infty} (y_n -x) \in X_+$. Since $x \in X_+$,
$ x=0$.

Assume that $X_+$ is normal. Then there exists some $c > 0$ such that
$\|y\| \le c \|x\|$ whenever $x,y \in X_+$ and $y \le x$.
Hence $\|x\| \le c \psi(x) $ for all $x \in X_+$. Set $\delta = 1/c$.
\end{proof}

\begin{theorem}[cf. {\cite[Thm.4.4]{KrLiSo}}]
\label{re:normal-equiv-norm}
Let $X$ be an ordered normed vector space with normal cone $X_+$.
Define
\[
\|x\|^\di = \max \{ \psi(x), \psi (-x) \}, \qquad x \in X,
\]
with $\psi$ from Proposition \ref{re:normal-mon-funct}.
Then $\|\cdot\|^\di$ is a norm on $X$ that is equivalent to
the original norm  (actually $\|x\|^\di \le \|x\|$ for all $x \in X_+$).
Further $\|\cdot\|^\di$ is order preserving on $X_+$: $\|x\|^\di \le \|y\|^\di$ for all $x, y \in X_+$
with $x \le y$. Finally, for all $x,y, z \in X$ with $x\le y\le z$,
\[
\|y\|^\di \le \max\{ \|x\|^\di, \|z\|^\di \}.
\]
\end{theorem}

\begin{proof} It is easy to see from the properties of $\psi$
that $\|\alpha x\|^\di = |\alpha| \|x\|^\di$
for all $\alpha \in \R$, $x\in X$, and that $\|\cdot \|^\di$ is subadditive.
Now $\psi(x) \le \|x\|$ and $\psi(-x) \le \|-x\|= \|x\|$ and
so $\|x\|^\di \le \|x\|$.

That $\|\cdot\|$ and $\|\cdot\|^\di$ are equivalent  norms
is shown in \cite[Thm.4.4]{KrLiSo}.

To prove the last statement, let $x,y,z \in X$ and $x \le y \le z$.
Then $y \le z $ and $-y \le -x$. Since $\psi$ is order preserving,
\[
\psi(y ) \le \psi (z)  \le \|z\|^\di,
\qquad
\psi(-y) \le \psi(-x) \le \|x\|^\di.
\]
This implies the assertion.
\end{proof}

\begin{corollary}
\label{re:compare}
Let $X$ be an ordered normed vector space with normal cone $X_+$.
Then there exists some $c \ge 0$ such that $\|y \| \le c\max\{ \|x\|, \|z\|\}$
for all $x,y,z \in X_+$ with $x \le y \le z$.
\end{corollary}

\begin{proof} Let $\|\cdot\|^\di$ be the equivalent norm from
Theorem \ref{re:normal-equiv-norm} and $c \ge 0$ such
that $\|x\|^\di \le \|x\| \le c \|x\|^\di$ for all $x \in X$.
Let $x \le y \le z$.
Then
\[
\|y\| \le c \|y\|^\di \le c\max\{ \|x\|^\di , \|z\|^\di \} \le
c \max\{ \|x\| , \|z\| \}. \qedhere
\]
\end{proof}

\begin{corollary}[Squeezing theorem {\cite[Thm.4.3]{KrLiSo}}]
\label{re:squeeze}
Let $X$ be an ordered normed vector space with
a normal  cone $X_+$.
Let $y \in X$ and $(x_n)$, $(y_n)$, $(z_n)$ be
sequences in $X$ with $x_n \le y_n \le z_n $ for all $n \in \N$
and $x_n \to y$ and $z_n \to y$. Then $y_n \to y$.
\end{corollary}

\begin{proof} Notice that $x_n - y \le y_n -y \le z_n -y$.
By Corollary \ref{re:compare}, with some $c\ge 0$ that does not depend on $n$,
\[
\|y_n - y\| \le c  \max\{\|z_n - y\|, \|x_n -y\|\} \to 0. \qedhere
\]
\end{proof}

\subsection{(Fully) regular cone}

\begin{definition}
\label{def:regular}
Let $X$ be an ordered normed vector space with cone $X_+$.

$X_+$ is called {\em regular} if any decreasing sequence
in $X_+$ converges.

$X_+$ is called {\em fully regular} if any increasing
bounded sequence in $X_+$ converges.

The norm of $X$ is called {\em additive} on $X_+$ if $\|x + z\|= \|x\| +\|z\|$
for all $x, z \in X_+$.

\end{definition}

\begin{theorem}
\label{re:regular}
 Let $X$ be an ordered normed vector space with cone $X_+$.
\begin{itemize}

\item[(a)] If $X_+$ is complete and regular, then $X_+$ is normal.

\item[(b)] If $X_+$ is complete and fully regular, then $X_+$ is normal.

\item[(c)] If $X_+$ is normal and fully regular, then $X_+$ is regular.

\item[(d)] If $X_+$ is complete and fully regular, then $X_+$ is regular.
\end{itemize}
\end{theorem}

\begin{proof} Notice that the proofs in \cite[1.5.2]{Kra} and \cite[1.5.3]{Kra} only
need completeness of $X_+$ and not of $X$.
\end{proof}

\begin{theorem}
Let $X$ be an ordered normed vector space with cone $X_+$.
If $X_+$ is complete with additive norm, then $X_+$ is fully regular.

\end{theorem}

\begin{proof}
Let $(x_n)$ be an increasing sequence in $X_+$ such that
there is some $c >0$ such that $\|x_n\| \le c$ for all $n \in N$.
Define $y_n= x_{n+1} - x_n$. Then $y_n \in X_+$ and
 $\sum_{k=j}^m y_n = x_{m+1} - x_j$. Since the norm is
 additive on $X_+$,
 \[
 \sum_{k=1}^m \|y_n\| = \Big \| \sum_{k=1}^m y_n \Big \|
 = \|x_{m+1} - x_1\| \le 2c, \qquad m \in \N.
 \]
This implies that $(x_n)$ is a Cauchy sequence in the complete
cone $X_+$ and thus converges.
\end{proof}

The standard cones of the Banach spaces $L^p(\Omega)$, $1 \le p < \infty$, are
regular and completely regular, while the cones of $BC(\Omega)$, the Banach
space of bounded continuous functions,  and of $L^\infty(\Omega)$
are neither regular nor completely regular though normal.

\subsection{The space of certain order-bounded elements
and some functionals}
\label{subsec:background-order-bounded}

\begin{definition}
\label{def:order-norm}
Let $x \in X$ and $u \in X_+$. Then $x$ is called $u$-{\em bounded} if there
exists some $c > 0$ such that $-cu \le x \le c u$. If $x$ is $u$-bounded, we define
\begin{equation}
\label{eq:order-max}
\|x\|_u = \inf \{c > 0; -cu \le x \le c u \}.
\end{equation}
The set of $u$-bounded elements in $X$ is denoted by $X_u$.
If $x,u \in X_+$ and $x $ is not $u$-bounded, we define
\[
\|x\|_u = \infty.
\]
Two elements $x$ and $u$ in $X_+$ are called {\em comparable} if $x$ is $u$-bounded
and $u$ is $x$-bounded.
\end{definition}

If $X$ is a space of real-valued functions on a set $\Omega$,
\[
\|x\|_u = \sup \Big \{ \frac{|x(\xi)|}{u(\xi)}; \xi \in \Omega, u(\xi) > 0\Big\}.
\]
Since the cone $X_+$ is closed,
\begin{equation}
\label{eq:order-max-for}
-\|x\|_u u \le x \le \|x\|_u u, \qquad x \in X_u.
\end{equation}
$X_u$ is a linear subspace of $X$,
$\|\cdot\|_u$ is a norm on $X_u$,
and $X_u$, under this norm, is an ordered normed vector space with
cone $X_+ \cap X_u$ which is normal, reproducing, and has nonempty
interior.

If $X_+$ is normal, by Theorem \ref{re:cone-normal}, there
exists some $M \ge 0$ such that
\begin{equation}
\label{eq:order-max-for-norms}
\|x\| \le M \|x\|_u \|u\|, \qquad x \in X_u.
\end{equation}
If $X_+$ is a  normal and complete cone of $X$, then
$X_+ \cap X_u$ is a complete subset of $X_u$ with the metric
induced by the  norm $\|\cdot\|_u$. For more information
see \cite[1.3]{Kra} \cite[I.4]{Boh},  \cite[1.4]{KrLiSo}.

For $u \in X_+$, one can also consider the functionals
\[
\left . \begin{array}{cc}
(x/u)^\di = & \inf\{ \alpha \in \R; x \le \alpha u \}
\\
(x/u)_\di = & \sup\{ \beta \in \R; \beta u \le x \}
\end{array} \right \} x \in X,
\]
with the convention that $\inf (\emptyset) = \infty$ and
$\sup (\emptyset) =- \infty$.
If $X$ is a space of real-valued functions on a set $\Omega$,
\[
\left . \begin{array}{cc}
(x/u)^\di = & \sup\big \{  \frac{x(\xi)}{u(\xi)}; \xi \in \Omega, u(\xi) >0 \big\}
\\[2mm]
(x/u)_\di = & \inf\big\{ \frac{x(\xi)}{ u(\xi)}; \xi \in \Omega, u(\xi) >0 \big\}
\end{array} \right \} x \in X.
\]
Many other symbols have been used for these two functionals
in the literature; see Thompson \cite{Tho} and Bauer \cite{Bau} for some
early occurrences.
For $x \in X_+$, $\|x\|_u =(x/u)^\di$.
Since we will use this functional for $x \in X_+$ only, we will
stick with the notation $\|x\|_u$.
Again for $x \in X_+$, $(x/u)_\di$ is a nonnegative real number,
and we will use the leaner notation
\begin{equation}
\label{eq:order-min}
[x]_u = \sup \{\beta \ge 0; \beta u \le x \}, \qquad x, u \in X_+.
\end{equation}
Since the cone $X_+$ is closed,
\begin{equation}
\label{eq:order-min-form}
x \ge [x]_u u, \qquad x,y \in X_+ .
\end{equation}
Further $[x]_u$ is the largest number for which this inequality holds.
The functional $[\cdot]_u: X_+ \to \R_+$ is homogeneous and
concave.



\section{More on homogeneous order-preserving maps and their spectral radii}

If the cones have appropriate properties, homogeneous order-preserving maps
are automatically bounded.

\begin{theorem}
\label{re:ordpres-poshom-boun}
Let $X$ and $Y$ be ordered normed vector spaces with cones $X_+$ and $Y_+$.
Let $X_+$ be a  complete or fully regular cone, $Y_+$ a normal cone,  and
$B: X_+ \to Y_+$ be homogeneous and order-preserving.
Then $B$ is continuous at 0 and in particular bounded.
\end{theorem}

The proof is adapted from \cite[Thm.2.1]{KrLiSo}.
\begin{proof}
Assume that $B$ is not continuous at 0. Then there exists some $\epsilon > 0$ and a sequence
$(x_n) \in X_+$ such that $\|x_n\| \le 2^{-2n} $ and $\|B x_n \| \ge \epsilon $
for all $n \in \N$. Since $X_+$ is complete or fully regular, the
series $\sum_{n=1}^\infty 2^n x_n=: w $ converges in $X_+$. Then $0 \le 2^n x_n \le w $ for
all $n \in \N$. Since $B$ is positively homogeneous and order-preserving,
$0 \le Bx_n \le 2^{-n} B w$. Since $Y_+$ is normal, $B x_n \to 0$. This contradiction shows
that $B$ is continuous at 0 and easily shown to be bounded.
\end{proof}

\begin{corollary}
\label{re:ordpres-poshom-boun}
Let $X$  be an ordered normed vector space with a normal cone $X_+$
that is  complete or fully regular. Let
$B: X_+ \to X_+$ be homogeneous and order-preserving.
Then $B$ is continuous at 0 and in particular bounded.
\end{corollary}

\subsection{Monotonicity of the cone and orbital spectral radii}

If the cone $X_+$ is normal, the cone and orbital spectral radius
are increasing functions of the homogeneous bounded order-preserving maps
(cf. \cite[L.6.5]{AGN})

\begin{theorem}
\label{re:spec-rad-increasing}
Let $X$ be an ordered normed vector space with a normal cone
$X_+$ and $A,B: X_+ \to X_+$ be bounded and  homogeneous.
Assume that $ Ax \le B x $ for all $x\in X_+$ and that $A$ or $B$
are order preserving. Then $\br_+(A) \le \br_+ (B)$ and $\br_o(A) \le \br_o(B)$.
\end{theorem}

\begin{proof}  We claim that $A^n x \le B^n x$ for all $x \in X_+$
and all $n \in \N$. For $n =1$, this holds by assumption.
Now let $n \in \N$ and assume the statement holds for $n$.
If $A$ is order-preserving,
then, for all $x \in X_+$, since $B^n x \in X_+$,
\[
A^{n+1} x = A A^n x \le A B^n x  \le B B^n x = B^{n+1} x.
\]
If $B$ is order-preserving,
then, for all $x \in X_+$, since $A^n x \in X_+$,
\[
A^{n+1} x = A A^n x \le B A^n x  \le B B^n x = B^{n+1} x.
\]
Since $X_+$ is normal, there exists some $c > 0$ such that
$\| A^n x \| \le c \|B^n x \| $ for all $x \in X_+$, $n \in \N$.
Thus, $\|A^n \|_+ \le c \|B^n \|_+$ for all $n \in \N$ and
$\br_+(A) \le \br_+(B)$. Further $\gamma_A(x) \le \gamma_B(x)$
and so $\br_o(A) \le \br_o(B)$.
\end{proof}

For a bounded positive linear operator on an ordered Banach space,
the spectral radius and cone spectral radius coincide provided that
the cone is reproducing \cite[Thm.2.14]{MPNu}. This is not true if the cone is only total
\cite[Sec.2,8]{Bon58}.

\begin{proposition}
\label{re:spec-rad=cone-spec-rad}
Let $X$ be an ordered Banach space and $X_+$ be  reproducing.
Then any linear positive operator $B: X \to Y$ that is bounded on $X_+$ is bounded.
Further  there exists some $c \ge 1$ such that, for all bounded linear
positive maps $B$ on $X$,
\[
\|B\|_+ \le \|B\| \le c \|B\|_+.
\]
If $X = Y$, $\br_+(B) = \br (B)$.
\end{proposition}

\subsection{A uniform boundedness principle}

\begin{theorem}
\label{re:unif-bounded}
Let $X_+$ be a normal complete cone of a normed vector space.
Let $\{B_j ; j \in J\}$ be an indexed family of continuous
homogeneous order preserving maps $B_j : X_+ \to X_+$. Assume that, for each $x \in X_+$,
$\{B_j (x) ; j \in J\}$ is a bounded subset of $X_+$.
Then $\{ \|B_j \|_+; j \in J\}$ is a bounded subset of $\R$.
\end{theorem}

\begin{proof}
By assumption,
\[
X_+ = \bigcup_{n \in \N} M_n
\qquad
M_n = \bigcap_{j \in J} \tilde M_{n,j},
\qquad
\tilde M_{n,j} = \{ x \in X_+; \|B_j(x) \| \le n \}.
\]
Since each $B_j$ is continuous, $\tilde M_{n,j}$ is a closed subset of $X_+$
for all $n,j \in \N$. Then $M_n$ is a closed subset of $X_+$ as an
intersection of closed sets. Since $X_+$ is complete by assumption,
by the Baire category theorem,
there exists some $n \in \N$ such that $M_n $ has nonempty interior.
So there exists some $z \in X_+$ and $\epsilon > 0$
such that
\[
z + \epsilon y \in M_n \qquad \hbox{whenever} \quad y \in X, z + \epsilon y \in X_+,
\|y\| \le 1.
\]
Since $z + \epsilon y \in X_+$ if $y \in X_+$,
\[
\|B_j (z+ \epsilon y ) \| \le n ,  \qquad  y \in X_+,
\|y\| \le 1, j \in J.
\]
Let $y \in X_+$, $\|y\| \le 1$,  $j \in J$. Since $B_j$ is homogeneous
and order preserving,
\[
\epsilon B_j (y)= B_j(\epsilon y) \le B_j (z + \epsilon y).
\]
Since $X_+$ is normal, there exists some $c \ge 0$ (independent of $y$ and $j$)
such that
\[
\|\epsilon B_j (y) \| \le c \|B_j (z+ \epsilon y)\| \le c n.
\]
Thus
\[
\|B_j (y) \| \le \frac{cn}{\epsilon}, \qquad y \in X_+, \|y\|\le 1, j \in J.
\]
By definition of $\|\cdot \|_+$,
\[
\|B_j \|_+ \le \frac{cn}{\epsilon}, \qquad j \in J. \qedhere
\]
\end{proof}


\subsection{A left resolvent}
\label{subsec:left-res}

There is one remnant from the usual relations between the spectral radius
and the spectrum of a linear operator that also holds in the homogeneous
case, namely that real numbers  larger than the
spectral radius are in the resolvent set. However, in the homogeneous case,
there only exists a left resolvent.

Let $X$ be a normed vector space with a  cone $X_+$
which is  complete or fully regular.
Let $B: X_+ \to X_+$ be homogeneous, continuous and order preserving.

For $\lambda > \br_+(B)$, we introduce $R_\lambda : X_+ \to X_+$,
\begin{equation}
\label{eq:left-resolvent}
R_\lambda (x) = \sum_{n=0}^\infty \lambda^{-n-1} B^n (x), \qquad x \in X_+.
\end{equation}
The convergence of the series follows from the completeness or
the full regularity of the cone.
Then $R_\lambda$ acts as a left resolvent,
\begin{equation}
\label{eq:left-resolvent-job}
\begin{split}
R_\lambda (B x) = &  \sum_{n=0}^\infty \lambda^{-(n+1)} B^{n+1}x
=
\sum_{n=1}^\infty \lambda^{-n} B^{n}x
\\
= &
\lambda R_\lambda ( x) - x, \qquad x \in X_+.
\end{split}
\end{equation}
It follows from the Weierstra\ss\ majorant test that the  convergence
of the series is uniform for $x$ in bounded subsets of $X_+$.
With this in mind, the following is easily shown.
\begin{lemma}
For $\lambda > \br_+(B)$, $R_\lambda$ is defined, homogeneous,
continuous, and order preserving.
\end{lemma}

The next result suggests that  $\br_+(B)$ is not an element
of the resolvent set of $B$  if $B$ is also superadditive.

\begin{theorem}
\label{re:resolvent-blowup}
Let $X$ be an ordered normed vector space with a normal cone $X_+$
that is complete or fully regular.
Let $B$ be homogeneous, order-preserving, bounded and $B(x+y) \ge B (x) + B(y)$
for all $x,y \in X_+$. Assume that $\br= \br_+(B) > 0$.
Then $\|R_{\lambda}\|_+ \to \infty$
as $\lambda \to \br+$.
\end{theorem}

Since $B$ is homogeneous, the superadditivity assumption for $B$ is
equivalent to $B$ being concave: $B((1-\alpha) x + \alpha y)
\ge (1-\alpha) B(x) + \alpha B(y)$ for $\alpha \in (0,1)$.
The proof is adapted from \cite{Bon58} and allows  $B$ to be concave rather than
additive.

\begin{proof}
 Since $X_+$ is normal, $\|R_\lambda\|_+$ is a decreasing function
of $\lambda > \br:=\br_+(B)$. $R_\lambda$ inherits superadditivity from $B$,
\begin{equation}
\label{eq:resolv-superadd}
R_\lambda (x+y) \ge R_\lambda (x) + R_\lambda (y), \qquad x, y \in X_+.
\end{equation}

Suppose that the assertion is false.
 Then there exists some $M \ge 0$ such that $\|R_\lambda \|_+ \le M$
for all $\lambda > \br$.
Let $0<\mu < \br < \lambda$ and $(\lambda - \mu) M < 1$. Since $X_+$ is complete
or fully regular,
\[
E_\mu (x) = \sum_{k=1}^\infty (\lambda - \mu)^{k-1} R_\lambda^k x
\]
converges for each $x \in X_+$. By (\ref{eq:left-resolvent-job}),
\[
E_\mu (\lambda  x) =  \sum_{k=1}^\infty (\lambda - \mu)^{k-1} R_\lambda^k (\lambda  x)
=
\sum_{k=1}^\infty (\lambda - \mu)^{k-1} R_\lambda^{k-1} (R_\lambda (Bx) + x).
\]
Since $R_\lambda$ is superadditive,
\[
\begin{split}
E_\mu (\lambda  x)
\ge & \sum_{k=1}^\infty (\lambda - \mu)^{k-1} R_\lambda^{k} (Bx)
+
\sum_{k=1}^\infty (\lambda - \mu)^{k-1} R_\lambda^{k-1} ( x)
\\
=& E_\mu(Bx ) + x + (\lambda - \mu ) E_\mu(x).
\end{split}
\]
Since $E_\mu$ is homogeneous,
\[
E_\mu x \ge \frac{1}{\mu} x + \frac{1}{\mu} E_\mu Bx, \qquad x \in X_+.
\]
By iteration and induction,
\[
E_\mu x \ge \sum_{k=0}^{n} \frac{1}{\mu^{k+1}} B^k x + \frac{1}{\mu^{n+1}} E_\mu B^{n+1}x , \qquad x \in X_+, n \in \N.
\]
This implies that
\[
E_\mu (x) \ge \frac{1}{\mu^{k+1}} B^k x, \qquad x \in X_+, n \in \N.
\]
Since $X_+$ is normal, there exists some $\tilde M > 0$ (which depends on $\mu$)
such that
\[
\|B^k x \| \le \tilde M \mu^k \|x\|, \qquad k \in \N, x \in X_+.
\]
Thus
\[
\|B^k\|_+ \le \tilde M \mu^k, \qquad k \in \N.
\]
Since $\mu < \br = \br_+(B)$, this is a contradiction.
\end{proof}

\begin{corollary}
\label{re:resolvent-blowup2}
Let $X$ be an ordered normed vector space with a complete normal cone $X_+$.
Let $B$ be homogeneous, order-preserving, continuous and $B(x+y) \ge B (x) + B(y)$
for all $x,y \in X_+$. Assume that $\br= \br_+(B) > 0$.
Then there exist some $x \in X_+$ and $x^* \in X_+^*$
such that $x^* ( R_{\lambda}x)  \to \infty$
as $\lambda \to \br+$.
\end{corollary}

\begin{proof}
By the uniform boundedness principle in Theorem \ref{re:unif-bounded},
there exists some $x \in X_+$ such that $\{ \|R_\lambda x\|; \lambda > \br_+(B)\}$
is unbounded. By the usual uniform boundedness principle, applied to
the Banach space $X^*$, there exists some $x^* \in X^*$
such that $\{ |x^*( R_\lambda x)|; \lambda > \br_+(B)\}$
is unbounded. Since $X_+$ is normal, $X_+^*$ is reproducing and the
unboundedness holds for some $x^* \in X_+^*$. Since $x^*(R_\lambda x)$
is a decreasing function of $\lambda$, $x^* ( R_{\lambda}x)  \to \infty$
as $\lambda \to \br_+(B)$.
\end{proof}

\subsection{Pointwise root power approximations of the spectral radius}
\label{sec:root-power}


In \cite{BaDa},  it is suggested for a linear positive operator $B$ to approximate $\br (B)$ by
$\|B^n x\|^{1/n}$, $x \in X_+, x\ne0$, for
 a special class of operators, but no general assumptions are specified within
 the paper that make this
 procedure work. Whenever this is possible for
 a homogeneous bounded map $B$ on $X_+$, $\br_+(B) = \gamma_B(x) = \br_o(B$.

We can assume without loss of generality that $\br_+(B) >0$, otherwise
this equality trivially holds for all $x \in X_+$.

\begin{definition}
\label{def:u-bounded}
Let $B: X_+ \to X_+$, $u \in X_+$. $B$ is called {\em pointwise $u$-bounded}
if for any $x \in X_+$ there exists  some $n \in \N$
and $\gamma > 0$ such that $B^n x \le \gamma u$.

$B$ is called {\em
uniformly $u$-bounded} if there exists
some $c > 0$ such that  $ B x \le c \|x\| u$ for all
$x \in X_+$. The element $u$ is called an {\em order bound} of $B$.

$B$ is called {\em uniformly order bounded} if it is uniformly $u$-bounded
for some $u \in X_+$.
$B$ is called {\em pointwise order bounded} if it is pointwise $u$-bounded for
some $u \in X_+$.

\end{definition}

This terminology has been adapted from various works by Krasnosel'skii
\cite[Sec.2.1.1]{Kra} and
coworkers  \cite[Sec.9.4]{KrLiSo} though it has been modified.

If $B: X_+ \to X_+$ is bounded and $X_+$ is solid, then $B$
is uniformly $u$-bounded for every interior point $u$ of $X_+$.

\begin{proposition}
\label{re:u-bounded-pointw-unif}
 Let $X$ be a normed  vector space with complete
cone $X_+$ and $B: X_+ \to X_+$ be continuous, order preserving and  homogeneous.
 Let $u \in X_+$ and $B$ be pointwise $u$-bounded.
Then some power of $B$ is uniformly $u$-bounded.
\end{proposition}

\begin{proof} Define
\[
M_{n,k} = \{ x \in X_+; B^n x \le k u \}, \qquad n,k \in \N.
\]
Since $B$ is continuous and $X_+$ is closed,
 each set $M_{n,k}$ is closed. By assumption, $X_+ = \bigcup_{k,n\in \N}
M_{n,k}$. Since $X_+$ is a complete metric space, by the Baire
category theorem, there exists some $n,k \in \N$ such that
$M_{n,k}$ contains a relatively open subset of $X_+$: There
exists some $y \in X_+$ and $\epsilon >0$ such that
$y + \epsilon z \in M_{n,k}$ whenever $z \in X$, $\|z\|\le 1$,
and $y + \epsilon z \in X_+$. Now let $z \in X_+$ and $\|z\|\le 1$.
Since $B$ is order preserving and $y +\epsilon z \in X_+$, $ B^n (\epsilon z) \le B^n( y +  \epsilon z) \le k u$.
Since $B$ is  homogeneous, for all $x\in X_+$, $x\ne 0$,
\[
B^n x = \frac{\|x\|}{\epsilon}  B^n \Big (\frac{\epsilon}{\|x\|} x \Big) \le
\frac{k}{\epsilon} \|x\| u. \qedhere
\]
\end{proof}

\begin{theorem}
\label{re:approx-power-order-bounded}
Let $X$ be a normed  vector  space with a  normal cone $X_+$
and $B: X_+ \to X_+$ be a  homogeneous, bounded,  order-preserving operator.
Let $u \in X_+$.

\begin{itemize}
\item[(a)] Assume
that $B^m$ is uniformly $u$-bounded for some $m \in \Z_+$.
Then
\[
\br_+(B) = \lim_{n\to \infty} \|B^n u\|^{1/n}.
\]
In particular, $\br_+(B) = \gamma_B(u) = \br_o(B)$.

\item[(b)] Assume that $B$ is pointwise $u$-bounded.
Then
\[
\br_o(B) = \lim_{n\to \infty} \|B^n u\|^{1/n}=\gamma_B(u).
\]

\end{itemize}

\end{theorem}

The case $m=0$ in part (a) (the identity map is uniformly $u$-bounded) is proved in \cite[Thm.2.2]{MPNu02}
 under the
additional assumption that $B$ is continuous.

\begin{proof} (a)
Since $B^m$ is uniformly $u$-bounded, there exists some $\gamma > 0$ such that
\[
  B^m x \le \gamma \|x\|  u, \qquad x \in X_+.
\]
Since $B$ is  homogeneous and order-preserving, for all $n \in \N$,
\[
0 \le  B^{m+n} x \le \gamma \|x\| B^n u, \qquad x \in X_+.
\]
Since $X_+$ is normal, there exists some $\tilde \gamma  >0$ such that
$\| B^{m+n} x \| \le \tilde \gamma \|x\| \,\|B^n u \|$ for all $n \in \N$, $x \in X$.
So $\|B^{m+n}\|_+^{1/n} \le \tilde \gamma^{1/n} \|B^n u\|^{1/n}$.
We can assume that $\br_+(B) >0$. Then $\liminf_{t\to \infty} \|B^n u\|^{1/n}
\ge \liminf_{n \to \infty} \|B^n \|_+^{1/n}= \br_+(B)$.

The proof of part (b) is similar.
\end{proof}

We also obtain  upper estimates of the cone spectral radius.

\begin{corollary}
\label{re:spec-rad-upper-est}
Let $X$ be a normed vector  space with a  normal cone $X_+$
and $B: X_+ \to X_+$ be a  homogeneous, bounded, order-preserving operator. Let $u \in X_+$, $\alpha \in \R_+$ and $k \in \N$
such that $ B^k u \le \alpha^k u$.

\begin{itemize}
\item[(a)]
If $B^m$ is uniformly $u$-bounded for some $m \in \N$,
then $\br_+(B) \le \alpha$.

\item[(b)] If  $B$ is pointwise $u$-bounded, then $\br_o(B) \le \alpha$.
\end{itemize}
\end{corollary}

\begin{proof} Since $B$ is order-preserving and  homogeneous,
$B^{k \ell}u \le \alpha^{k\ell} u $ for all $\ell \in \N$. Since $X_+$ is normal,
there exists some $c>0$ such that $\|B^{k \ell}u \|\le c \alpha^{k\ell} \|u\| $ for all
$\ell \in \N$.
 Now apply Theorem \ref{re:approx-power-order-bounded}.
\end{proof}

Lower estimates of the cone spectral radius can be obtained under less
assumptions. In particular, the cone does not need to be normal.
Notice that the proof of \cite[L.2.2]{LeNu} works if $B$ is just bounded
and order preserving but not necessarily continuous.
\begin{theorem}
\label{re:spec-rad-lower-est} Let $X$ be a normed vector space with  cone $X_+$.
Let $B: X_+ \to X_+$ be  homogeneous, bounded, and order preserving.
Further let $x \in X_+$, $m \in \N$, and $\alpha \ge 0$ such that $B^m x \ge \alpha^m x$.
Then $\br_+(B) \ge \br_o(B) \ge \alpha$.
\end{theorem}


\section{Eigenvectors for compact maps}

Recall the large-time bound of the geometric growth factor for initial value $u \in X_+$,
\begin{equation}
\gamma_B(u):= \limsup_{n\to \infty } \|B^n u \|^{1/n}.
\end{equation}

\begin{proposition}
\label{re:pert-eigen-comp-prep}

Let $X$ be an ordered normed vector space with  normal cone $X_+$
  and $B: X_+ \to X_+$ be  homogeneous,
order preserving, and bounded.
Let $K: X_+ \to X_+$ be continuous and compact and $K(x) \ge Bx$
for all $x \in X_+$, $\|x\|=1$. Finally let $u \in X_+$, $u \ne 0$.

Then there exists some $x \in X_+$, $\|x\|=1$,
and some $\lambda \ge \gamma_B(u)$ such that $Kx +u  = \lambda x$.
\end{proposition}

The idea of perturbing $K$ and using one of the classical
fixed point theorems seems as old as the theory of positive operators
\cite[Sec.2.2]{Kra} \cite[Sec.3]{Sch55}.

\begin{proof}
Define
\[
K_u (x) =  K(x) +   u , \qquad x \in X_+ .
\]
Then $K_u (x ) \ge  u $ and $\{\|K_u (x)\|;x \in X_+\}$ is bounded away from 0
 because $X_+$ is closed. We can define
\[
\tilde K (x) = \frac{K_u (x)}{\|K_u (x) \|}, \qquad x \in X_+.
\]
Then $\tilde K$ is continuous and maps the set $C = X_+ \cap \bar U_1$, $\bar U_1 = \{x\in X; \|x\| \le 1\}$, into itself. $C$ is a closed convex subset of $X$ and $\tilde K (C)$
has compact closure. By Tychonov's fixed point theorem \cite[Thm.10.1]{Dei},
 $\tilde K$ has a fixed point $x \in X_+$, $\|x\|=1$,
\[
K (x)   +u = \lambda x, \qquad \lambda = \|K (x) +u \|>0.
\]
Since $K(x ) \in X_+$, $ u \le \lambda x  $. By assumption, $Bx \le \lambda x$.
Since $B$ is order preserving and  homogeneous, $B^n x \le \lambda^n x$
for all $n \in \N$. Hence $B^n u \le \lambda^{n} \lambda x $. Since $X_+$ is normal
and $\|x\|=1$,
there exists some $\delta_x > 0$ such that $\lambda^n \ge \delta_x \|B^n u\|$
for all $n \in \N$. Hence $\lambda \ge \delta_x^{1/n} \|B^n u\|^{1/n} $ for all $n \in \N$.
This implies that $\lambda \ge \gamma_B (u)$.
\end{proof}

Recall that the orbital spectral radius is defined by $\br_o(B) =\sup_{x\in X_+}
\gamma_B(x)$.

\begin{theorem}
\label{re:pert-eigen-comp}
 Let $X$ be an ordered normed vector space with  normal cone $X_+$
  and $B: X_+ \to X_+$ be  homogeneous,
order preserving, and bounded, $\br_o(B) > 0$.

Let $K: X_+ \to X_+$ be continuous and compact and $K(x) \ge Bx$
for all $x \in X$, $\|x\|=1$.

Then there exists some $x \in X_+$, $\|x\| = 1$,
and some $\lambda \ge \br_o(B)$ such that $Kx = \lambda x$.
If $B = K$, then $\lambda = \br_o(B)$.
\end{theorem}

\begin{proof}
Since $\gamma_B(\alpha u) = \gamma_B (u)$ for all $u \in X_+$, $\alpha >0$,
we can choose a sequence $(u_n)$ in $X_+$ such that $\|u_n\| \to 0$
and $\gamma_B(u_n) \to \br_o(B)$.
Define
\[
K_n (x) =  K(x) +   u_n , \qquad x \in X_+, n \in \N .
\]
By Proposition \ref{re:pert-eigen-comp-prep}, for each $n \in \N$, there exist $x_n
\in X_+$ such that $\|x_n\|=1$ and $\lambda_n \ge \gamma_B(u_n)$ such that
\[
\lambda_n x_n = K_n (x_n) = K(x_n) + u_n.
\]
 Since $K$ is compact, after choosing a subsequence,
$K x_n$ converges to some $y \in X_+$.
The corresponding subsequence of $(\lambda_n)$ is bounded
and $\lambda_n \to \lambda \ge \br_o(B)>0$ after choosing another subsequence.
Since $u_n \to 0$,  $x_n \to x \in X_+$ with $\|x\|=1 $ and $K x = \lambda x$. If
$B =K$,
$\br_o(B) \ge \gamma_B(x) = \lambda$  and so
$\br_o(B) = \lambda$.
\end{proof}

\begin{remark}
\label{re:pert-eigen-comp-rem}
Let $B =K$. Then $\br_o(B) = \br_+(B)$ by Theorem \ref{re:spec-rad-alt-char} (ii).
Further, if $X$ is a Banach space,  normality of the cone does not need to be assumed
(\cite[Thm.2.1]{Nus81} and the remarks in \cite{LeNu}).
\end{remark}


\section{Monotonically compact maps}


\begin{definition}
\label{def:mono-compact}
Let $u\in X_+$, $u \ne 0$, and  $B: X_+ \to X_+$ be pointwise $u$-bounded. $B$
is called {\em monotonically compact}
if $(Bx_n)$ converges for each monotone sequence $(x_n)$ in $X_+$ for which
there is some $c >0$ such that $x_n \le c u$ for all $n \in \N$.
\end{definition}

If $X_+$ is regular (Definition \ref{def:regular}),
then every continuous $B: X_+ \to X_+$ is
monotonically compact.

\begin{definition}
 \label{def:minihedral}
 $X_+$ is called {\em minihedral} if $ x \land y = \inf \{x,y\}$
exists for any $x, y \in X_+$.
\end{definition}

\begin{theorem}
\label{re:minihed-subeigen}
 Let $X$ be an ordered normed vector space with
a  normal minihedral cone $X_+$. Let $B: X_+ \to X_+$ be continuous, order preserving,
and  homogeneous. Further let  $B$ be monotonically compact and
uniformly $u$-bounded for some $u \in X_+$, $u\ne 0$. Finally assume
that $\br_+(B) > 0$.

Then there exists some $x \in X_+$, $x \ne 0$,
such that $B x \ge \br_+ (B) x$ and $\br_o(B) = \br_+(B)$.

\end{theorem}

The first part of the proof follows \cite[L.9.5]{KrLiSo} almost verbatim where $B$ is assumed
to be a linear operator on the ordered Banach space $X$. It is given
here for the ease of the reader and the author's peace of mind
that linearity can be safely dropped.

\begin{proof}
Let $u \in X_+$ such that $B$ is uniformly $u$-bounded.
Since $B$ is  homogeneous, we can assume that $\br_+(B) =1$ and $u \ne 0$.
We define
\[
x_0= u, \qquad x_k = \min\{ B x_{k-1} + 2^{-k} u, u \}, \quad k \in \N.
\]
Then $x_k \le u=x_0$ for all $n\in \N$. By induction, since $B$ is order
preserving, $x_{k+1}\le x_{k}$ for all $k \in \N$. Since $B$ is monotonically
compact, there exists some $z \in X_+$ such that $B x_k \to z$ as $k \to \infty$ and $Bx_k \ge z$ for all $k \in \N$.
Then $y_k := B x_{k-1} + 2^{-k} u \to z $. Further $y_k \ge B x_{k-1} \ge z$.
Thus,
\[
x_k = \min\{ y_k, u \} \ge \min\{z, u\}=:x.
\]
 Notice that
\[
\min\{y_k, u\} + z - y_k \le \min \{z, u\}=x.
\]
So $0 \le x_k -x  \le y_k -z$. Since $X_+$ is normal and $y_k \to z$,
also $x_k \to x$. Since $B$ is continuous, $B x =z \ge x$.
Further
$
x = \min\{z, u \} = \min \{Bx, u\}.
$

It remains to show that $x \ne 0$. Suppose $x_k \to 0$ as $k \to \infty$.
Since $B$ is uniformly $u$-bounded, there exists some $m \in \N$ such
that $B x_{k-1} + 2^{-k} u  \le u$ for all $k \ge m$. Hence
\[
x_k = B x_{k-1} + 2^{-k} u  , \qquad k \ge m.
\]
In particular, $2^m x_m \ge u$ and $x_k \ge B x_{k-1}$ for all $k \ge m$. Since $B$ is order preserving
and  homogeneous,
\[
2^m x_{m +n} \ge B^n (2^m x_m ) \ge B^n u
\]
and
\[
2^m B x_{m+n} \ge B^{n+1} u.
\]
Since $x_{m+n} \to 0$ as $n \to \infty$ and $B$ is $u$-bounded,
$2^m B x_{m+n} \le (1/2) u$ for sufficiently large $n$. This shows
that, for some $n \in \N$, $ B^{n+1} u \le (1/2) u$.
By Corollary \ref{re:spec-rad-upper-est},
 $\br_+ (B) < 1$, a contradiction.

 This shows that $x \ne 0$ and $Bx \ge x$. By (\ref{eq:growth-factor}),
 $\gamma_B(x) \ge 1$
 and $r_o(B) \ge 1= \br_+(B)$ by (\ref{eq:spec-rad-orb}).
\end{proof}

The following definition is similar to the one in \cite[III.2.1]{Boh}

\begin{definition}
An order preserving map $B: X_+ \to X_+$ is called
strictly increasing if for any $x, y \in X_+$ with
$x \le y$ and $\|x\| \ne \|y\|$ there exists some $\epsilon > 0$
and some $m \in \N$ such that $B^m (y) \ge (1+\epsilon) B^m x$.
\end{definition}

\begin{theorem}
\label{re:minihed-eigen}
 Let $X$ be an ordered normed vector space with
a  normal minihedral cone $X_+$. Let $B: X_+ \to X_+$ be continuous, strictly
increasing,
and  homogeneous. Further let  $B$ be  monotonically compact
 and
uniformly $u$-bounded for some $u \in X_+$, $u\ne 0$.
Finally assume
that $\br_+(B) > 0$.

Then there exists some $x \in X_+$, $x \ne 0$,
such that $B x = \br_+ (B) x$.
\end{theorem}

\begin{proof}
We can assume that $\br_+(B) =1$.
By Theorem  \ref{re:minihed-subeigen}, there exists some $x \in X_+$,
$\|x\|=1$ such that $B x \ge  x$. Then the sequence $(x_n)_{n \in \Z_+}$
in $X_+$ defined by $x_n = B^n x$ is increasing. We claim that $\|x_n\| =1$
for all $n \in \N$. Suppose that $n \in \N$ is the first $n$ with
$\|x_n\| \ne  1$. Then $\|x_{n-1} \|\ne \|x_n\|$. Since $B$ is strictly
increasing, there exists some $\epsilon > 0$ and some
$m \in \N$ such that $B^m x_n \ge (1+\epsilon) B^m x_{n-1}$.
By definition of $(x_n)$,
$B y \ge (1+\epsilon) y $ for $y = B^m x_{n-1} \ge x$.
Since $y \in X_+$ and $y \ne 0$, $\br_+(B) \ge 1+ \epsilon$, a
contradiction.

Set $x_0 =x$. Then $x_n = B x_{n-1}$ for $n \in \N$.
Since $B$ is uniformly $u$-bounded, there exists some $c \ge 0$
such that $x_n \le c \|x_{n-1}\| u = cu$. Since $B$ is monotonically
compact,
 $(Bx_{n-1}) = (x_{n}) $ converges
to some $y \in X_+$, $\|y\| =1$. Since $B$ is continuous, $By =y$.
\end{proof}

\section{Eigenfunctionals}
\label{sec:dual}

The celebrated Krein-Rutman theorem does not only state the existence
of a positive eigenvector but also of a positive eigenfunctional
of a positive linear map provided that the map is compact
and the cone is total or that the cone is normal and solid. We explore what still can be done if the additivity of the operator is dropped.
Recall the left resolvents $R_\lambda$, $\lambda > \br_+(B) > 0$,
in Section \ref{subsec:left-res}.

\begin{proposition}
\label{re:eigenfun-prep}
 Let $X$ be an ordered normed vector space
with cone $X_+$ and $B: X_+ \to X_+$ be homogeneous,
order preserving, and uniformly order bounded. Assume that $\br= \br_+(B) >0$
and that there exists $v \in X_+$ and $x^* \in X^*_+$ such that
$x^* (R_\lambda (v)) \to \infty$ as $\lambda \to \br_+$.

Then there exists a homogeneous,
order preserving, bounded nonzero functional $\phi: X_+ \to \R$
such that $ \phi \circ  B  = \br \phi$
and $\phi(u) > 0$ for each order bound $u$ of $B$.
\end{proposition}

Our proof will not provide  continuity of $\phi$.

\begin{proof}
Let $x^*$ and $v$ be as above. Choose a sequence $(\lambda_n)$
in $(\br, \infty)$ such that $x^* (R_{\lambda_n} (v)) \to \infty$
as $n \to \infty$. Define $\psi_n : X_+ \to \R_+$ by
\[
\psi_n(x) = x^* (R_{\lambda_n} (x)), \qquad n \in \N, x \in X_+.
\]
The functionals $\psi_n$ are homogeneous, order preserving,
and continuous, $\psi_n (v) \to \infty$ as $n \to \infty$.
By (\ref{eq:left-resolvent-job}),
\[
\psi_n(B x) = x^* ((R_{\lambda_n} (Bx))= x^* (R_{\lambda_n} (\lambda_n x)- x)
=
\lambda_n \psi_n (x) - x^* x.
\]
We set
\[
\phi_n = \frac{\psi_n }{\|\psi_n\|_+}.
\]
Then $\|\phi_n\|_+ =1 $ and
\begin{equation}
\label{eq:dual-converge}
 \lambda_n \phi_n- \phi_n \circ B = \frac{x^*}{\|\psi_n\|} \to 0, \qquad n \to \infty.
\end{equation}
By Tychonoff's compactness theorem for topological products, there exists some
\[
\phi \in \bigcap_{m\in \N} \overline{  B_m },
\qquad
B_m = \{ \phi_n; n \ge m\},
\]
where the closure is taken in the topology of pointwise convergence
on $\{x\in X_+; \|x\| \le 1 \}$. Notice that all $\phi_n$
are order preserving, bounded and homogeneous. $\phi$ inherits these
properties. For instance, let $x_1, x_2 \in X_+$ and $x_1 \le x_2$. Then
there exist a strictly increasing sequence $(n_j)$ of natural
number such that $\phi_{n_j} (x_i) \to \phi (x_i) $ as $j \to \infty$, $i=1,2$.
Since $\phi_{n_j} (x_1) \le \phi_{n_j} (x_2)$ for all $j \in \N$,
also $\phi (x_1)\le \phi(x_2)$.

Similarly, for $x \in X_+$,
there exists a strictly increasing sequence $(n_j)$ of natural
number such that $\phi_{n_j} (x) \to \phi (x) $ and $\phi_{n_j} (Bx)
\to \phi(Bx)$ as $j \to \infty$. By (\ref{eq:dual-converge}),
$\phi (Bx) = \br \phi(x)$.

We need to rule out that $\phi$ is the zero functional.
Let $u \in X_+$, $\|u\|=1$, such that $B$ is uniformly $u$-bounded:
There exist some $c \ge 0$ such that $B x \le c \|x\| u$
for all $x \in X_+$.

Let $x\in X$, $\|x\| \le 1$. Since each $\phi_n$ is order preserving,
by (\ref{eq:dual-converge}),
\[
\lambda_n \phi_n( x)    = \phi_n ( Bx) + \frac{ x^*(x)}{\| \psi_n \|_+}
\le
\phi_n (c \|x\| u) + \frac{ x^*(x)}{\| \psi_n \|_+}
\le
c \phi_n ( u)  +  \frac{\|x^*\|}{\|\psi_n\|_+} .
\]
Since this holds for all $x \in X_+$, $\|x\|\le 1$,
\[
\lambda_n = \lambda_n \|\phi_n\|_+ \le
c \phi_n (u) + \frac{\|x^*\|}{\|\psi_n\|_+} .
\]
Since $\phi_{n_j} (u) \to \phi (u)$ for some strictly increasing
sequence $(n_j)$ in $\N$ and $\|\psi_{n_j}\|_+ \to \infty$,
\[
0< \br \le c \phi(u) . \qedhere
\]

\end{proof}

\begin{theorem}
\label{re:dual}
 Let $X$ be an ordered normed vector space
with cone $X_+$ and $B: X_+ \to X_+$ be homogeneous,
order preserving, and uniformly order bounded. Assume that $\br= \br_+(B) >0$
and that there exists some $v \in X_+$, $v \ne 0$, such
that $B v \ge \br v$. Then there exists a homogeneous,
order preserving, bounded nonzero functional $\phi: X_+ \to \R$
such that $ \phi \circ B  = \br \phi$.
\end{theorem}

\begin{proof}
Let $(\lambda_n)$ be a sequence in $(\br, \infty)$ such that
$\lambda_n \searrow \br$ as $n \to \infty$. Let $v \in X_+$, $\|v\|=1$
and $B v \ge \br v$.
Choose some $x^* \in X_+^*$ such that $x^* x > 0$ \cite[Sec.1.4.1]{Kra}.
Then
\[
x^* (R_\lambda (v)) = \sum_{k=0}^\infty \lambda^{-(k+1)} x^* (B^k v )
\ge
\sum_{k=0}^\infty \lambda^{-(k+1)} x^* (\br^k v )
=
\frac{x^*v}{\lambda - \br} \stackrel{\lambda\to \br+} {\longrightarrow}
\infty.
\]
Apply Proposition \ref{re:eigenfun-prep}.
\end{proof}

The following result is well-known for vectors rather than
functionals if $B$ is linear \cite[Thm.9.3]{KrLiSo}{\cite[Thm.2.2]{Nus}}.

\begin{lemma}
\label{re:eigen-power-riddance}
 Let $X$ be an ordered normed vector space
with cone $X_+$ and $B: X_+ \to X_+$. Let $\br > 0$, $p \in \N$, and $\psi:
X_+ \to \R$ with $\psi \circ B^p = \br^p \psi$. Then
$\varphi \circ B = \br \varphi$ for
\[
\varphi= \sum_{k=0}^{p-1} \br^{-k}  \psi \circ B^k .
\]
\end{lemma}

\begin{proof}
For all $x \in X_+$,
\[
\varphi(B(x)) = \sum_{k=0}^{p-1} \br^{-k} \psi (B^{k+1} (x))=
\sum_{k=1}^{p-1} \br^{1-k} \psi (B^{k} (x)) + \br^{1-p} r^p \psi (x)
=
\br \varphi(x) .
\]
\end{proof}

\begin{corollary}
\label{re:dual-power}
 Let $X$ be an ordered normed vector space
with cone $X_+$ and $B: X_+ \to X_+$ be homogeneous, bounded
and order preserving. Assume that $\br= \br_+(B) >0$
and that there exist $m,\ell \in \N$ and some $v \in X_+$, $v \ne 0$, such
that $B^m v \ge \br^m v$ and $B^\ell$ is uniformly order bounded. Then there exists a homogeneous,
order preserving, bounded nonzero functional $\phi: X_+ \to \R$
such that $ \phi \circ  B  = \br \phi$.
\end{corollary}

\begin{proof} Set $p = m\ell$. Then $B^p \ge \br^p v$ and $B^p$ is uniformly
order bounded. By Theorem \ref{re:dual}, there exists some
homogeneous,
order preserving, bounded nonzero functional $\varphi: X_+ \to \R$
such that $ \varphi ( B^p(x))  = \br^p \varphi(x)$ for all $x\in X_+$.
Apply the previous lemma and notice that $\varphi$ inherits
the desired properties from $\psi$ and $B$.
\end{proof}

\begin{theorem}
\label{re:eigenfun-exist-summary}
Let $X$ be an ordered normed vector space with
  normal cone $X_+$ and $B: X_+ \to X_+$ be homogeneous, order
preserving,  continuous, and some power of $B$ is uniformly order bounded, $\br= \br_+(B) > 0$. Then
there exists some homogeneous, order preserving, bounded
$\phi: X_+ \to R_+$ with $\phi \circ B = \br \phi$ if at least
one of the following assumptions is satisfied.

\begin{itemize}

\item[(a)] A power of $B$ is compact.

\item[(b)] $X_+$ is a minihedral cone, and some power of $B$ is monotonically
compact.

\end{itemize}
\end{theorem}

\begin{proof} (a) Combine Theorem \ref{re:pert-eigen-comp} with Corollary
\ref{re:dual-power}.

(b) Combine Theorem \ref{re:minihed-eigen} with Lemma \ref{re:eigen-power-riddance}.
\end{proof}

\begin{theorem}
\label{re:eigenfun-exist}
Let $X$ be an ordered normed vector space with a normal complete cone $X_+$.
Let $B$ be homogeneous, order-preserving,  pointwise order bounded and $B(x+y) \ge B (x) + B(y)$
for all $x,y \in X_+$.
 Assume that $\br= \br_+(B) >0$.
Then there exists a homogeneous, concave,
order preserving, bounded nonzero functional $\phi: X_+ \to \R$
such that $ \phi ( B(x))  = \br \phi(x)$ for all $x\in X_+$.
\end{theorem}

\begin{proof}
By Theorem \ref{re:u-bounded-pointw-unif}, we can assume
that some  power of $B$ is uniformly order bounded.
By Lemma \ref{re:eigen-power-riddance}, we can assume that
$B$ itself is uniformly order bounded.
 By Corollary \ref{re:resolvent-blowup2}, $x^*( R_\lambda ( x)) \to
\infty$ as $\lambda \to \br+$ with some $x\in X_+$, $x^* \in X_+^*$.
Apply Proposition \ref{re:eigenfun-prep}.
\end{proof}


\section{Scent of a beer barrel}
\label{subsec:beer}
We show existence of eigenvectors for two classes of order-preserving
homogeneous continuous maps $B$ which are not necessarily compact
themselves but have compact powers.

\begin{theorem}
\label{re:beer-exist}
Let $X$ be an ordered normed vector space with a
cone $X_+$ and $B:X_+ \to X_+$ be homogeneous, order-preserving and continuous.
Assume that $X_+$ is normal or $X$ a Banach space.

Let $Y$ be a linear subspace of $X$, $B(X_+) \subseteq Y$,  which carries a norm $\|\cdot\|_Y$
such that  $\{\|B x\|_Y; \; x \in X_+, \, \|x\|\le 1\}$ is bounded
and  $B$ is compact from $(Y_+, \|\cdot\|_Y)$ into $(X_+, \|\cdot\|)$
where $Y_+ = X_+ \cap Y$.

Assume that $\br_+(B) > 0$ and that there is some  $u \in Y_+$, $u \ne 0$, such that   $B$ is uniformly $u$-bounded.

Then there exists some $v \in X_+ \cap Y$, $\|v\|=1$, such that $Bv = \br_+(B) v$.
\end{theorem}

The normality of $X_+$ can be dropped if $X$ is an ordered Banach space.

\begin{lemma}
\label{re:beer-exist-prep}
Let $X$ be an ordered normed vector space with a
cone $X_+$ and $B:X_+ \to X_+$ be homogeneous, order-preserving and continuous.
Assume that $X_+$ is normal or $X$ a Banach space.

Let $\psi: X_+ \to \R_+$ that is homogeneous, order preserving, continuous and
$\psi (x) >0$ for all $x\in X_+$, $x \ne 0$. Let $u \in X_+$
and $B$ be uniformly $u$-bounded.

 Set  $Q(x)  = \psi(x) u$ and $\tilde B = B + Q$.

Then the following hold.
\begin{itemize}

\item[(a)]
$\tilde B$ is continuous, homogeneous and order preserving.

\item[(b)]
For all $x, y \in X_+$ with $x \le y$ and $\psi(x) < \psi (y)$,
there exists some $\eta > 0$ such that $\tilde By \ge (1+\eta) \tilde Bx$.

Let  $B$ satisfy the assumptions of Theorem \ref{re:beer-exist}. Then

\item[(c)]  $\tilde B^2 $ is compact.

\item[(d)] There exists  a unique $v\in X_+$ such that $\psi(v)=1$ and
$\tilde B v = \br_+(B) v $.

\end{itemize}
\end{lemma}

\begin{proof} (a) Notice that $Q$ is continuous, homogeneous and order preserving.

(b)
We can assume that $x \ne 0$.
Since $B$ is uniformly $u$-bounded, there exists some $c \ge 0$ such that
$Bx \le c\|x\| u$ for all $x \in X_+$. Let $\delta = \frac{\psi(y) - \psi(x) }{2}$.
Then
\[
\begin{split}
\tilde B y = &
B x
+  \psi (x)u
+  2 \delta u
\\
\ge &
Bx +  \psi(x) u +  \delta \frac{\psi(x)} {\|\psi\|_+ \|x\|} u
+
\frac{ \delta}{(c+1)\|x\|}  B x
\ge
\eta \tilde B (x)
\end{split}
\]
with $\eta > 0 $ being the smaller of $\frac{\delta  }{\|\psi\|_+ \|x\|}$
and $\frac{\delta}{(c+1)\|x\|} $.

(c) Let $(x_k)$ be a bounded sequence in $X_+$. By assumption $(B x_k)$ is
 a bounded sequence in $(Y_+, \|\cdot\|_Y)$. Since $u \in Y \cap X_+$,
$(Q(x_k))$ is a bounded sequence in $Y$  and so is $(\tilde B x_k)_k$.
By the compactness assumption for $B$, $ (B (\tilde B x_k))_k$ has a subsequence
converging in $X_+$. Notice that $Q$ is compact from $X_+$ to $Y$ with
the stronger norm. After choosing suitable subsequences,
$ (B (\tilde B x_k))_k$ and $(Q (\tilde B x_k))_k$ converge and so does $(\tilde B^2 x_k)_k$.

(d) Since $\tilde B u \ge  \psi(u) u$, $\br_o(\tilde B) \ge  \psi (u) > 0$
by Theorem \ref{re:spec-rad-lower-est}.
We can assume that $\br_o(\tilde B) =1$.
By Theorem \ref{re:pert-eigen-comp} and Remark \ref{re:pert-eigen-comp-rem},
there exists some $v\in X_+$ such that $\psi(v)=1$ and
$\tilde B^2 v =  v $.
Assume that there is some $w \in X_+$ with
the same properties.

By construction of $\tilde B$,
\[
v \ge  \psi( Bv +  \psi(v)u) u
\ge
\psi (\psi(v)u) u = \psi(v) \psi(u) u.
\]
Since $\psi(u)>0$ and $\psi(v)>0$, $u$ is $v$-bounded. Since $B$ is uniformly $u$-bounded,
\[
v \le c \|\tilde B v\| u ,
\]
and $v$ is $u$-bounded. The same holds for $w$ such that $v$ and $w$ are comparable.
By construction, $\tilde B$ is uniformly $u$-bounded. So $\tilde B$ is uniformly
$v$-bounded and $\br_+(\tilde B) = \br_o(\tilde B)$.

Recall the functional $[\cdot]_v$ from (\ref{eq:order-min}).
Since $w$ is $v$-comparable, $w \ge [w]_v v$ with $[w]_v > 0$.

Suppose that $\psi (w) > \psi([w]_v v)$.
By part (b), there exists some $\eta > 0$ such that
\[
\tilde B w \ge (1+\eta) \tilde B ([w]_v v).
\]
Then
\[
w = \tilde B^2 w \ge   (1+\eta) [w]_v \tilde B^2  v = (1+\eta) [w]_v v.
\]
By definition, $[w]_v \ge (1+\eta) [w]_v$, a contradiction.

This implies $\psi(w) = \psi([w]_v v) = [w]_v \psi (v). $
Since $\psi(v) =1 = \psi(w)$, $[w]_v =1$ and $w \ge v$.

By symmetry, $v \ge w$ and $v = w$.

Now set $w = \tilde  Bv$. Since $\tilde B w =v\ne 0 $, $w\ne 0$. Then $\tilde B^2 w = w$.
By our previous consideration $\frac{1}{\psi(v)} v = \frac{1}{\psi(Bv)} Bv$.
So there exists some $\lambda > 0$
such that $\tilde Bv = \lambda v$. Then $v= \tilde B^2 v = \lambda \tilde B v = \lambda^2 v.$
This implies $\lambda =1$ and $\tilde B v =v $.
\end{proof}

\begin{proof}[Proof of Theorem \ref{re:beer-exist}]
Since $X_+$ is normal, by Proposition \ref{re:normal-mon-funct},
there is a homogeneous, subadditive, order-preserving functional
$\psi: X_+ \to \R_+$ such that $\psi(x)>0$
for all $x \in X_+$, $x\ne 0$.

Choose a sequence $(\epsilon_n)$ of positive numbers
such that $\epsilon_n \to 0$. Set $Q_n (x) = \epsilon_n \psi(x) u$
and $B_n = B + Q_n$.

By Lemma \ref{re:beer-exist-prep}, for each $n \in \N$ there exists a unique $v_n \in X_+$
such that $\psi(v_n) =1 $ and $B_n v_n = \lambda_n v_n$, $\lambda_n = \br_+(B_n)$.

By Theorem \ref{re:spec-rad-increasing}, $\lambda_n \ge \br_+(B)$ and $\lambda_n
\le \br_+(B + \eta \psi(\cdot) u)$ for all $n \in N$, where $\eta = \sup_{n\in \N}
\epsilon_n$.
After choosing a subsequence we can assume that $\lambda_n \to \lambda \ge \br_+(B)>0$.

 $(B v_n)$ is a bounded sequence in $(Y_+, \|\cdot\|_Y)$ and so is $(B_n v_n)$
 and thus $(\lambda_n v_n)$. Notice that
\[
\lambda_n^2 v_n = B (\lambda_n v_n) + \epsilon_n \psi(\lambda_n v_n) u.
\]
By the compactness assumption for $B$, $(B (\lambda _n v_n))$ converges after choosing a
subsequence. Since $\epsilon_n \to 0$ and $\lambda_n \to \lambda>0$, we can assume that $(v_n)$ converges to some
$w \in X_+$, $\|w\|=1$. Since $B$ is continuous, $B w = \lambda w$ and $\lambda
=\br_+(B)$ by Theorem \ref{re:spec-rad-lower-est}.
\end{proof}

$B$ is called {\em essentially compact} if there exists some $k \in \N$ such
that $B^k = K + L$ with a continuous compact operators $K : X_+
\to X_+$ and a linear operator bounded operator $L:X \to X$ and $\br(L) < [\br_+(B)]^k$.

\begin{theorem}
\label{re:beer-mono-comp1}
 Let $X$ be a normed vector space with a normal and minihedral cone $X_+$.
Let $B: X_+ \to X_+$ be homogeneous, continuous, order preserving,
and essentially compact, $\br_+(B) >0$. Assume that there is some $u \in X_+$,  $u \ne 0$, such that
 $B$ is uniformly $u$-bounded  and monotonically compact.
Further assume the following one-sided uniform continuity condition for $B$:

\begin{quote}
For any $\epsilon > 0$ and any $c > 0$ there exists some $\delta > 0$
such that $\|B (x+y) - B(x) \|_u \le \epsilon $ for all $x,y \in X_+ \cap X_u$
with $\|x\|_u \le c$ and $\|y\|_u \le \delta$.
\end{quote}

Then there exists some $v \in X_+$, $v \ne 0$, such that $Bv =\br_+(B)v$.
\end{theorem}

\begin{proof}
We can assume that $\br_+(B)=1$.
Since $X_+$ is normal, there exists an equivalent monotone  norm. Let $\psi$
be the restriction of that norm to $X_+$ and $(\epsilon_n)$
be a sequence of positive numbers such that $\epsilon_n \searrow  0$.
 Define $B_n: X_+ \to X_+$
by $B_n (x) = Bx + \epsilon_n \psi(x) u$. Then $B_n$ is continuous, homogeneous,
order preserving, uniformly $u$-bounded and monotonically compact.
$B_n$ also satisfies Lemma \ref{re:beer-exist-prep} (b),
so $B_n$ is strictly increasing. By Theorem \ref{re:minihed-eigen},
 there exist eigenvectors
$v_n  \in X_+$, $\psi(v_n)=1$, such that
\begin{equation}
\label{eq:beer-mono-comp-eigen}
\lambda_n v_n = B_n (v_n)= B (v_n ) + \epsilon_n u,  \qquad \lambda_n =\br_+(B_n).
\end{equation}
By Theorem \ref{re:spec-rad-increasing}, $(\lambda_n)$ is a decreasing sequence and
$\lambda_n \ge 1 $. So $\lambda_n \to \lambda$ with $\lambda \ge 1$.
Notice that $\lambda_n v_n \ge B v_n$. Thus, for all $k \in \N$,
$\lambda_n^k v_n \ge B^k v_n$.

Since $\psi$ is the restriction of an equivalent norm to $X_+$,
the sequence $(v_n)$ is bounded. Since $B$ is uniformly $u$-bounded,
there is some $c \ge 0$ such that
\[
v_n \le \frac{1}{\lambda_n} ( B v_n + \epsilon_n u)
\le
c \|v_n\| u  + \epsilon_1 u, \qquad n \in \N.
\]
So $(v_n)$ is a $u$-bounded sequence in $X_u \cap X_+$.
By induction,
\begin{equation}
\label{eq:beer-mono-comp-eigen2}
B^k_n v_n \le B^k v_n + \alpha_n u , \qquad \alpha_n \to 0.
\end{equation}
This is true for $k=1$. Assume that $k \in \N$ and it holds for $k$. Then
\[
B^{k+1}_n v_n \le  B (B^k v_n + \alpha_n u) +  \epsilon_n \psi( B^k v_n + \alpha_n u) u
\]
Now $(x_n)$ with $x_n= B^k v_n$ is a $u$-bounded sequence in $X_u \cap X_+$.
Further, the restriction of $B$ to $X_u \cap X_+$ with $u$-norm is assumed
to be uniformly continuous on every $u$-bounded subset set of $X_u \cap X_+$.
\[
\begin{split}
B^{k+1}_n v_n \le &  B^{k+1} v_n + B (x_n + \alpha_n u) - B(x_n) +  \epsilon_n \psi( B^k v_n + \alpha_n u) u
\\
\le &
 B^{k+1} v_n + \|B (x_n + \alpha_n u) - B(x_n)\|_u u  + \epsilon_n [\psi( B^k v_n ) + \alpha_n \psi(u) ] u.
\end{split}
\]
So $B_n^{k+1} v_n \le B^{k+1} v_n + \tilde \alpha_n u$ with
\[
\tilde \alpha_n  = \|B (x_n + \alpha_n u) - B(x_n)\|_u  + \epsilon_n [\psi( B^k v_n ) + \alpha_n \psi(u) ] .
\]
Since $B$ satisfies the uniform $u$-continuity condition assumed
above, $\tilde \alpha_n \to 0$.

By (\ref{eq:beer-mono-comp-eigen}) and (\ref{eq:beer-mono-comp-eigen2}),
\[
B^k v_n \le \lambda_n^k v_n \le B^k v_n + \alpha_n u, \qquad \alpha_n \to 0.
\]

Since $B$ is essentially compact, there exists some $k \in \N$ such
that $B^k = K + L$ with a compact continuous operator $K : X_+
\to X_+$ and a linear bounded operator $L:X \to X$ and $\br(L) <  [\br_+(B)]^k=1$.
So
\[
K v_n \le \lambda_n^k v_n - L v_n \le K v_n + \alpha_n u, \qquad \alpha_n \to 0.
\]
Recall that $\lambda_n \searrow \lambda \ge 1$. We rearrange,
\[
K v_n + (\lambda^k - \lambda_n^k)v_n  \le (\lambda^k  - L) v_n \le K v_n + \alpha_n u +
(\lambda^k - \lambda_n^k)v_n , \qquad \alpha_n \to 0.
\]
Since $(v_n)$ is a bounded sequence, after choosing a subsequence,
$K v_n \to z$, $n \to \infty$, and both the left and the right hand side converge
to $z$. By the squeezing theorem (Corollary \ref{re:squeeze}),
$(\lambda^k - L)v_n \to z$. Since $\lambda^k \ge 1 > \br(L)$
and $(\lambda^k - L)$ has a continuous inverse, $v_n \to (\lambda^k - L)^{-1} z=:v$.

Since $\lambda_n \to \lambda \ge 1$, by (\ref{eq:beer-mono-comp-eigen}),
 $\psi (v) =1$ and $Bv =\lambda v$. Since $\br_+(B) =1$, $\lambda =1$.
\end{proof}

\begin{theorem}
\label{re:beer-mono-comp2}
 Let $X$ be a normed vector space with a normal and minihedral cone $X_+$.
Let $B: X_+ \to X_+$ be homogeneous, continuous, order preserving,
essentially compact, $\br_+(B) >0$. Assume that there is some $u \in X_+$,  $u \ne 0$, such that
 $B$ is uniformly $u$-bounded  and monotonically compact.
Further assume the following one-sided uniform continuity condition for $B$:

\begin{quote}
For any $\epsilon > 0$ and any $c > 0$ there exists some $\delta > 0$
such that $\|B (x+y) - B(x) \| \le \epsilon $ for all $x,y \in X_+ $
with $\|x\| \le c$ and $\|y\| \le \delta$.
\end{quote}

Then there exists some $v \in X_+$, $v \ne 0$, such that $Bv =\br_+(B)v$.
\end{theorem}

\begin{proof}
As in the proof of Theorem \ref{re:beer-mono-comp1},
there exist eigenvectors
$v_n  \in X_+$, $\psi(v_n)=1$, such that
\begin{equation}
\lambda_n v_n = B_n (v_n)= B (v_n ) + \epsilon_n u,  \qquad \lambda_n =\br_+(B_n).
\end{equation}
By induction, $B^k_n v_n = B^k v_n +  u_n  $ with $u_n \in X_+$ and $u_n \to 0$.

Obviously this holds $k=1$. Assume that $k \in \N$ and the statement holds
for $k$. Then
\[
\begin{split}
B_n^{k+1} v_n = & B (B^k v_n +  u_n ) + \epsilon_n \psi (B^k v_n +  u_n ) u
\\
= &
B^{k+1} v_n + B(x_n +u_n) - B(x_n)  + \epsilon_n \psi(x_n + u_n) u
\end{split}
\]
where $(x_n)$ is the bounded sequence $x_n= B^k v_n $.
By the uniform continuity condition for $B$ and $u_n \to 0$, $\epsilon_n \to 0$,
\[
B_n^{k+1} v_n = B^{k+1} v_n + w_n,
\qquad
w_n = B(x_n +u_n) - B(x_n)  + \epsilon_n \psi(x_n + u_n) u \to 0.
\]
The remainder of the proof is the same as for Theorem \ref{re:beer-mono-comp1}.
\end{proof}

\section{Application to  a spatially distributed two-sex population}
\label{sec:two-sex}

The population we consider has individuals of both sexes which form pairs in order to reproduce. Most two-sex population models are formulated in continuous time
\cite{Had93, Had08, IMM}. Here we consider the case that the mating occurs
once a year and that the mating
season is short which makes a discrete-time model more appropriate.
We also assume that individuals do not live to see two mating seasons.

The spatial habitat of the population is represented by a Borel set $\Omega \subseteq \R^m $. If  $f: \Omega \to \R_+$ is an integrable function (with respect to the Lebesgue measure),
$f(\xi) $, $\xi \in \Omega$, represents the  number of newborns at $\xi \in \Omega$.



\subsection*{The migration operators}

In order to take account of the movements of individuals over the year, we consider  integral operators $K_j$, $j=1,2$,
\begin{equation}
\label{eq:migration-op}
(K_j f) (\xi) = \int_\Omega k_j(\xi , \eta) f(\eta) d \eta, \qquad
f \in \cM_+, \xi \in \Omega, j=1,2.
\end{equation}
Here
$k_j(\xi,\eta)\ge 0$ gives the rate at which individuals  that  are born at $\eta$
are female ($j=1$) or  male ($j=2$) and  will be at $\xi$ in the year after.
$\cM_+$ denotes the set of nonnegative Borel measurable functions on $\Omega$.

In fact, $\int_\Omega k_j(\xi, \eta) d \xi$ is the probability that
an individual born at $\eta$ is female or male, respectively, and
will be still alive in the year after.

\bigskip

Assume $k_j: \Omega^2 \to \R_+ $ to be   Borel measurable and
\begin{equation}
\int_\Omega k_j(\xi, \eta) d \xi \le 1, \qquad j=1,2, \quad \eta \in \Omega.
\end{equation}
Then the  $K_j$ are bounded linear operators on $L^1(\Omega)$
and $\|K_j f\|_1 \le \|f\|_1$ for all $f \in L^1(\Omega)$.

These assumptions are assumed throughout this section without
further mentioning.



\subsection*{The mating and birth operator}

The mating and birth operator, $F: \R_+^\Omega \times \R_+^\Omega \to R_+^\Omega$, is defined by
\[
F(f,g) (\xi) = \phi (\xi, f(\xi), g(\xi)), \qquad f,g \in \R_+^\Omega, \xi \in \Omega.
\]
Here $\R_+^\Omega$ is the set of  functions on $\Omega$ with values in $\R_+$
and $\phi : \Omega \times \R_+^2 \to \R_+$ is the local mating
and birth function. If there are $x_1$ females and $x_2$ males
at location $\xi$ in $\Omega$, $\phi(\xi, x)$ with $x = (x_1,x_2)$ is the amount of offspring produced at $\xi$. We equip $\R_+^2$ with the standard order
and make the following assumptions which hold throughout the rest of this
paper:

\begin{itemize}
\item $\phi (\cdot, x)$ is Borel measurable for each $x \in \R_+^2$.

\item $\phi(\xi, \cdot)$ is order preserving on $\R_+^2$ for each $\xi \in \Omega$.

\item $ \phi(\xi, \cdot)$ is  homogeneous for each $\xi \in \Omega$,
\[
\phi(\xi , \alpha x ) = \alpha \phi (\xi, x), \qquad \alpha \ge 0, \xi \in \Omega,
x \in \R_+^2.
\]

\item $\phi(\xi, \cdot)$ is continuous for each $\xi \in \Omega$.

\item The function $\psi: \Omega \to \R_+ $ defined by $\psi(\xi) =\phi(\xi, 1,1) $ is  bounded.

\end{itemize}

One example is given by the harmonic mean
\[
\phi(\xi, x) =
 \beta(\xi) \frac{x_1 x_2 }{x_1 + x_2},
\quad x= (x_1,x_2) \in \R_+ \setminus\{(0,0)\}.
\]
Here $\beta:\Omega \to \R_+$ is Borel measurable, and $\beta(\xi)$ is per pair  birth rate  at $\xi$.
Another example is
\[
\phi(\xi, x) = \min \big \{ \beta_1(\xi) x_1, \beta_2(\xi) x_2 \big\}
\]
with two Borel measurable functions $\beta_1,\beta_2: \Omega \to \R_+$.

Notice that
$F:\R_+^\Omega \times \R_+^\Omega \to \R_+^\Omega$ is  homogeneous and order-preserving. Here $\R_+^\Omega$ is equipped with with the
pointwise order $f \le g$ if $f(\xi) \le g(\xi)$ for all $\xi \in \Omega$.
 $F$ has only weak positivity and order-preserving properties:
If can happen that   $f,g $ are not identically zero but $F(f,g)$ is if the supports of $f$ and $g$
have empty intersections. $F$ is the Nemytskii or superposition operator
associated with $\phi$.

For $x = (x_1,x_2) \in \R_+^2$,
\begin{equation}
\label{eq:mate-birth-est}
\phi(\xi,x_1,x_2)\le \phi(\xi, x_1 +x_2 , x_1+x_2)
= \psi(\xi) (x_1 + x_2) , \qquad \psi(\xi):= \phi(\xi,1,1).
\end{equation}
So
\begin{equation}
\label{eq:mate-birth-op-est}
F(f,g) \le  \psi (f+g), \qquad f,g \in \R_+^\Omega, \qquad \psi(\xi) = \phi(\xi,1,1).
\end{equation}
It follows from the boundedness of $\psi$ that $F$ is a continuous
map from $L^p_+(\Omega) \to L^p_+(\Omega)$ for every $p \in [1,\infty]$.
See \cite[Thm.3.4.4]{GaPa}.

\subsection*{The next year offspring operator}

Our state space of choice is $X_+ = L_+^1(\Omega)$, the cone
of $X = L^1(\Omega)$.
The next year offspring operator is formally given by
\begin{equation}
\label{eq:next-year-op}
B (f) = F(K_1f, K_2f), \quad f \in L^1_+(\Omega).
\end{equation}
By (\ref{eq:mate-birth-est}),
\begin{equation}
\label{eq:next-year-est}
B (f) \le  \psi (K_1f + K_2 f) , \qquad f \in L^1_+(\Omega).
\end{equation}

In order to make the operator $B$ uniformly $u$-bounded for some $u \in X_+$ we make the following assumption (cf. \cite[(2.4)]{Kra}).

\begin{assumption}
\label{ass:beer-app}
Assume that there exists a function $u \in  L^1(\Omega)$
such that
\[
\psi (\xi) (k_1(\xi, \eta) + k_2(\xi, \eta)) \le u(\xi)
\]
for a.a. $(\xi, \eta) \in \Omega^2$   with respect to the $2n$-dimensional Lebesgue measure.
\end{assumption}

We will establish that $B$ maps
 $X_+$ into $X_u \cap X_+$ where $X_u$ is defined as in Definition    \ref{def:order-norm}.

\begin{lemma}
\label{re:beer-app1}
 $B$  maps  $X_+= L_+^1(\Omega)$ into $X_u \cap X_+$.

 Further $\{\|B f\|_u ; \; f \in X_+, \|f\|_1 \le 1\}$ is bounded
 and $B$ is uniformly $u$-bounded.
\end{lemma}

\begin{proof} For all $\xi \in \Omega$ and $f \in L^1_+(\Omega)$,
by (\ref{eq:mate-birth-est}) and (\ref{eq:migration-op}),
\[
\phi(\xi, (K_1f)(\xi), (K_2f)(\xi))
\le
\psi (\xi) \int_\Omega [k_1(\xi, \eta) + k_2(\xi,\eta)] f(\eta) d \eta
\]
with the right hand side being nonnegative and possibly infinite.
For all $g \in L^1_+(\Omega)$, by Tonelli's theorem and our assumptions,
\[
\begin{split}
&\int_\Omega g(\xi) \phi(\xi, (K_1f)(\xi), (K_2f)(\xi)) d\xi
\\
\le &
\int_{\Omega^2}  g(\xi) \psi (\xi)  [k_1(\xi, \eta) + k_2(\xi,\eta)]  f(\eta) d\xi d \eta
\\
\le &
 \int_{\Omega^2} g(\xi) u(\xi) f(\eta) d\xi d \eta
=
\int_\Omega  g(\xi) u(\xi ) \|f\|_1 d\xi.
\end{split}
\]
This shows that $(B f)(\xi) $ is defined for a.a. $\xi \in \Omega$
and that $Bf \le \|f\|_1 u$ a.e. on $\Omega$.
 So $B$ maps $X=L^1(\Omega)$
into $X_u$ and $\{\|Bf \|_u; f \in X_+, \|f\|_1 \le 1 \}$ is bounded by 1.
\end{proof}

Since $B$ is uniformly $u$-bounded and monotonically compact
($X_+$ is  regular), we have the following result
from Theorem \ref{re:approx-power-order-bounded}, Theorem \ref{re:minihed-subeigen} and Theorem \ref{re:eigenfun-exist-summary}.

\begin{theorem}
\label{re:mate-subeigen} Let the Assumptions \ref{ass:beer-app} hold. Then
\[
\|B^n u \|^{1/n} \to \br_+ (B), \quad  n \to \infty.
\]
Further there exists some $f \in L_+^1(\Omega)$, $v \ne 0$, such that
$B (f) \ge \br_+(B) f$. There also exists some homogeneous
order-preserving bounded
functional $\phi: X_+ \to \R_+$ such that $\phi (u) > 0$ and
$ \phi \circ B = \br_+(B) \phi$.

Finally $\br_+(B) = \br_o(B)$.
\end{theorem}

\subsection{An eigenvector with beer barrel scent }


 Our aim is applying  Theorem \ref{re:beer-exist} with $Y = X_u$ endowed with the $u$-norm.

\begin{lemma}
\label{re:beer-app2}
 $B$ is a compact map from $X_+ \cap X_u$
with the $u$-norm to $X_+= L^1_+(\Omega)$.
\end{lemma}

\begin{proof} Since $u \in L^1(\Omega)$, by Tonelli's theorem,
\[
\infty > \int_\Omega \Big (\int_\Omega k(\xi, \eta ) d \xi \Big ) u(\eta) d \eta
=
\int_\Omega d \xi \int_\Omega k(\xi, \eta ) u(\eta) d \eta .
\]
Thus, for a.a. $\xi \in \Omega$, $k_j (\xi, \cdot)
\in L^1_+( u d \eta):= L^1(\Omega, \mu) $
where $\mu $ is the finite regular measure
$\mu (S ) = \int_S u(\eta) d\eta $ on the Borel subsets $S$ of $\Omega$.
We extend $\mu$ to the Borel subsets $S$ of $\R^n$ by $\tilde \mu (S)
= \int_S u (\eta) d \eta$. $\tilde \mu$ is then a finite Baire measure
on the locally compact space $\R^n$.
Since $\R^n$ is $\sigma$-compact (countable at infinity) and its topology has a countable base,
  $C_c(\R^N)$, the space of
continuous functions with compact support,   is countable under the
 supremum norm \cite[Thm.7.6.3]{Bau} and dense in $L^1(\R^n, \tilde \mu)$
 \cite[Cor.7.5.6]{Bau}. So $L^1(\R^n, \tilde \mu)$ has a dense
 countable subset. Restriction to $\Omega$ provides a dense
 countable subset $\{g_i; i \in \N\}$  of $L^1(u d \eta)$.

Let $(f_\ell)$ be a bounded sequence in $X_u \cap X_+$ with the $u$-norm.
Then there exists some $c > 0$ such that $f_\ell \le c u$ for all $\ell  \in \N$.
For each $i \in \N$, $(\int_\Omega g_i(\eta) f_\ell(\eta) d\eta)_{\ell \in \N}$
is a bounded real sequence which has a convergent subsequence. Using a
standard diagonalization procedure, after choosing a subsequence of $(f_\ell)$,
we can assume that  $(\int_\Omega g_i(\eta) f_\ell(\eta) d\eta)_{\ell \in \N}$
are convergent sequences for all $i \in \N$. Let $g \in L^1(u d \eta)$
and $\epsilon > 0$. Then there exists some $i \in \N$ such that
$\int_\Omega |g(\eta)- g_i(\eta)| u (\eta) d \eta < \epsilon$.
For $\ell, m \in \N$,
\[
\begin{split}
& \Big | \int_\Omega g(\eta) f_\ell(\eta) d \eta -
\int_\Omega g(\eta) f_m(\eta) d \eta \Big |
\\
\le
&\Big | \int_\Omega g(\eta) f_\ell(\eta) d \eta - \int_\Omega g_i(\eta) f_\ell (\eta) d \eta\Big|
\\
&+
\Big |
\int_\Omega g_i(\eta) f_\ell (\eta) d \eta
-
\int_\Omega g_i(\eta) f_m (\eta) d \eta
\Big|
 \\+ &
\Big | \int_\Omega g_i(\eta) f_m (\eta) d \eta
-
\int_\Omega g(\eta) f_m(\eta) d \eta \Big |
\\
\le &
2c \int_\Omega | g(\eta) - g_i(\eta)| u(\eta) d\eta
 +
\Big |
\int_\Omega g_i(\eta) f_\ell (\eta) d \eta
-
\int_\Omega g_i(\eta) f_m (\eta) d \eta
\Big|.
\end{split}
\]
Since $(\int_\Omega g_i (\eta) f_\ell (\eta) d \eta)_\ell$
is a Cauchy sequence,
\[
\limsup_{\ell, m \to \infty} \Big | \int_\Omega g(\eta) f_\ell(\eta) d \eta -
\int_\Omega g(\eta) f_m(\eta) d \eta \Big |
\le
2c \int_\Omega | g(\eta) - g_i(\eta)| u(\eta) d\eta
\le 2c \epsilon.
\]
Since this holds for each $\epsilon >0$,
$\Big (\int_\Omega g(\eta) f_\ell(\eta) d \eta \Big)_\ell$
is a Cauchy sequence and converges for each
$g \in L^1(u d \eta)$. By the consideration of the beginning of this proof,
\[
(K_j (f_\ell) (\xi))_\ell \stackrel{\ell \to \infty}{\longrightarrow} h_j (\xi),
 \hbox{ for a.a. } \; \xi \in \Omega
 \]
with  nonnegative Borel measurable function $h_j : \Omega \to \R_+$, $j=1,2$.
Since $\phi(\xi, \cdot)$ is continuous for all $\xi \in \Omega$ ,
\[
B(f_\ell)(\xi)\stackrel{\ell \to \infty}{\longrightarrow}  F(h_1,h_2)(\xi)
\hbox{ for a.a. } \xi \in \Omega.
\]
Since $B$ is homogeneous and order preserving,
\[
B(f_\ell) \le c B (u) \in L^1_+(\Omega), \qquad \ell \in \N.
\]
Taking the limit as $\ell \to \infty$, $F(h_1,h_2) \le c B(u)$
and
\[
| B(f_\ell) - F(h_1,h_2)| \le 2c B(u).
\]
By Lebesgue's theorem of a.e. dominated convergence,
$\| B(f_\ell) - F(h_1,h_2) \|_1 \to 0$ as $\ell \to \infty$.

This shows that $B$ is compact as a map from $X_u \cap X_+$ with $u$-norm to $X_+$.
\end{proof}

\begin{theorem}
\label{re:beer-app-eigen}
 Let Assumption \ref{ass:beer-app}  be satisfied
and $\br_+(B) > 0$. Then there exists an eigenfunction $f \in
L^1_+ (\Omega)$, $f\ne 0$,
such that $B f = \br_+(B) f$.
\end{theorem}

\begin{proof}
The assertions follow from Theorem  \ref{re:beer-exist} with $Y = X_u$ whose assumptions have been checked in Lemma   \ref{re:beer-app1}
and Lemma \ref{re:beer-app2}. In particular, there exists some $f \in X_+$, $\|f\|_1=1$,
such that $Bf = \br_+(B) f$.
\end{proof}


\subsection*{Acknowledgment}

I thank Karl-Peter Hadeler for his constructive criticism of an earlier draft and historical
remarks.

 \bibliography{}

\begin{thebibliography}{10}


\bibitem{AGN} {\sc Akian, M., Gaubert, S., Nussbaum, R.D.},
A Collatz-Wielandt characterization of the spectral radius
of order-preserving homogeneous maps on cones,
arXiv:1112.5968v1

\bibitem{BaDa} {\sc Baca\"er, N., E.H. Ait Dads},
Genealogy with seaonality, the basic reproduction number,
and the influenza pandemic,
{\em J. Math. Biol.} {\bf 62} (2011),  741-762

\bibitem{Bau} {\sc Bauer, H.},
{\em Probability Theory and Elements of Measure Theory}, sec. ed.,
Academic Press, London 1981



\bibitem{Bir} {\sc Birkhoff, G.},
Extensions of Jentzsch's Theorem,
{\em Trans. Amer. Math. Soc.} {\bf 85} (1957), 219-226


\bibitem{Boh66} Bohl, E., Eigenwertaufgaben bei monotonen Operatoren
und Fehlerabsch\"atzungen f\"ur Operatorgleichungen,
{\em Arch. Rat. Mech. Anal. 22} (1966), 313-332

\bibitem{Boh} {\sc Bohl, E.},
{\em Monotonie: L\"osbarkeit und Numerik bei Operatorgleichungen},
Springer, Berlin Heidelberg 1974

\bibitem{Bon55} {\sc Bonsall, F.F.},
Endomorphisms of a partially ordered vector space without
order unit,
{\em J. London Math. Soc.} {\bf 30} (1955),
133-144


\bibitem{Bon58} {\sc Bonsall, F.F.},
Linear operators in complete positive cones,
{\em Proc. London Math. Soc.} {\bf 8} (1958),
53-75

\bibitem{Bon62} {\sc Bonsall, F.F.},
Lectures on Some Fixed Point Theorems of Functional Analysis,
Tata Institute of Fundamental Research, Bombay 1962

\bibitem{Col1} {\sc Collatz, L.},
Einschlie\ss ungssatz f\"ur die Eigenwerte von Integralglei\-chungen,
{\em Math. Z.} {\bf 47} (1942), 395-398

\bibitem{CuZh} {\sc Cushing, J.M., Y. Zhou},
The net reproductive  value and stability in
matrix population models,
{\em Nat. Res. Mod.} {\bf 8} (1994), 297-333


\bibitem{Dei} {\sc Deimling, K.D.}
{\em Nonlinear Functional Analysis},
Springer, Berlin Heidelberg 1985

\bibitem{DiHeMe}
{\sc Diekmann, O.; J. A. P. Heesterbeek, J. A. J. Metz},  On the definition and the computation of the basic reproduction ratio $R_0$ in models for infectious diseases in heterogeneous populations, {\em  J. Math. Biol.} {\bf  28}  (1990),   365-382


\bibitem{EvNu} {\sc Eveson, S.P., R.D. Nussbaum},
Applications of the Birkhoff-Hopf theorem to the spectral theory of
positive linear operators,
{\em Math. Proc. Camb. Phil. Soc.} {\bf 117} (1995), 491-512


\bibitem{FoNa} {\sc F\"orster, K.-H., B. Nagy},
On the Collatz-Wielandt numbers and the local spectral radius of
a nonnegative operator,
{\em Linear Algebra and its Applications} {\bf 120} (1980),
193-205


\bibitem{GaPa} {\sc Gasinski, L., N.S. Papageorgiou},
{\em Nonlinear Analysis}, Chapman\&Hall/CRC, Boca Raton 2004

\bibitem{Had89} {\sc Hadeler, K.P.},
Pair formation in age-structured populations,
{\em Acta Appl. Math.} {\bf 14} (1989), 91-102

\bibitem{Had93} {\sc Hadeler, K.P.}, Pair formation models with maturation period.  {\em  J. Math. Biol.}  {\bf 32}  (1993),   1-15.

\bibitem{Had08} {\sc Hadeler, K.P.},  Homogeneous systems with a quiescent phase.  {\em    Math. Model. Nat. Phenom.} {\bf 3}  (2008),   115-125.



\bibitem{IMM} {\sc Iannelli, M., M. Martcheva, F.A. Milner},
{\em Gender-Structured Population Models: Mathematical Methods,
Numerics, and Simulations},
SIAM, Philadelphia 2005


\bibitem{Kra} {\sc Krasnosel'skij, M.A.},
{\em Positive Solutions of Operator Equations,}
Noordhoff, Groningen 1964

\bibitem{KrLiSo} {\sc Krasnosel'skij, M.A., Lifshits, Je.A., Sobolev, A.V.}
{\em Positive Linear Systems: The Method of Positive Operators},
Heldermann Verlag, Berlin 1989

\bibitem{KrRu} {\sc Krein, M.G., M.A. Rutman},
Linear operators
leaving invariant a cone in a Banach space (Russian),
{\em Uspehi Mat. Nauk (N.S.)} {\bf 3} (1948), 3-95
English Translation, AMS Translation 1950 (1950), No. 26

\bibitem{LeNu} {\sc Lemmens, B., R.D. Nussbaum},
Continuity of the cone spectral radius, preprint
(arXiv:1107.4532v2)

\bibitem{MPNu02} {\sc Mallet-Paret, J., R.D. Nussbaum},
Eigenvalues for a class of homogeneous cone maps arising from max-plus operators,
{\em Discr. Cont. Dyn. Sys. (DCDS-A)} {\bf 8} (2002), 519-562

\bibitem{MPNu} {\sc Mallet-Paret, J., R.D. Nussbaum},
Generalizing the Krein-Rutman theorem, measures of noncompactness
and the fixed point index,
{\em J. Fixed Point Theory and Appl.} {\bf 7} (2010),
103-143

\bibitem{MPNu-beer} {\sc Mallet-Paret, J., R.D. Nussbaum},
Asymptotic fixed point theory and the beer barrel theorem,
{\em J. Fixed Point Theory Appl.} {\bf 4} (2008), 203-245



\bibitem{Nus81} {\sc Nussbaum, R.D.},
Eigenvectors of nonlinear positive operators and the linear
Krein-Rutman theorem, {\em Fixed Point Theory} (E. Fadell and G. Fournier, eds.),
309-331, Springer, Berlin New York 1981


\bibitem{Nus} {\sc Nussbaum, R.D.},
Eigenvectors of order-preserving linear operators,
{\em J. London Math. Soc.} {\bf 2} (1998), 480-496


\bibitem{Nus85} {\sc Nussbaum, R.D.},
{\em The Fixed Point Index and Some Applications,}
Les Presses de l'Universit\'e de Montr\'eal, Montr\'eal 1985

\bibitem{Nus87} {\sc Nussbaum, R.D.},
Iterated nonlinear maps and Hilbert's projective metric: a summary,
{\em Dynamics of Infinite Dimensional Systems} (S.-N. Chow, J.K. Hale, eds.),
231-248, Springer, Berlin Heidelberg 1987


\bibitem{Nus88} {\sc Nussbaum, R.D.},
Hilbert's projective metric and iterated nonlinear maps,
{\em Mem. AMS} {\bf 75}, Number 391, Amer. Math. Soc., Providence 1988


\bibitem{Nus89} {\sc Nussbaum, R.D.},
Iterated nonlinear maps and Hilbert's projective metric, II,
{\em Mem. AMS} {\bf 79}, Number 401, Amer. Math. Soc., Providence 1989


\bibitem{Sch55} {\sc Schaefer, H.H.},
 Positive Transformationen in lokalkonvexen halbgeordneten Vektorr\"aumen,
 {\em  Math. Ann.} {\bf 129},  1955, 323-329



\bibitem{Sch59} {\sc Schaefer, H.H.},
Halbgeordnete lokalkonvexe Vektorr\"aume. II.
{\em Math. Ann.} {\bf 138} (1959), 259-286


\bibitem{Sch} {\sc Schaefer, H.H.},
{\em Topological Vector Spaces}, Macmillan, New York 1966

\bibitem{Thi09} {\sc Thieme, H.R.},
Spectral bound and reproduction number for infinite
dimensional population structure and time-heterogeneity,
{\em SIAM J. Appl. Math.} {\bf 70} (2009), 188-211



\bibitem{Tho} {\sc Thompson, A.C.},
On certain contraction mappings in a partially ordered
vector space,
{\em Proc. AMS} {\bf 14} (1963), 438-443



\bibitem{Tro} {\sc Tromba, A.J.}, The beer barrel theorem, a new
proof of the asymptotic conjecture in fixed point theory,
{\em Functional Differential Equations and Approximations of
Fixed Points} (H.-O. Peitgen, H.-O. Walther, eds.), 484-488,
Lecture Notes in Mathematics 730, Springer, Berlin Heidelberg 1979

\bibitem{DrWa} {\sc van den Driessche, P.,  J. Watmough},
Reproduction
numbers and sub-threshold endemic equilibria for compartmental models
of disease transmission,
 {\em Math. Biosci.} {\bf 180} (2002), 29-48


\bibitem{vMi} {\sc von Mises, R.,  H. Pollaczek-Geiringer},
 Praktische
Verfahren der Gleichungsaufl\"osung,
{\em  Zeitschrift f\"ur
Angewandte Mathematik und Mechanik} {\bf 9} (1929), 58-77,
152-164,

\bibitem{Wie} {\sc Wielandt, H.},
Unzerlegbare, nicht negative Matrizen,
{\em Math. Z.} {\bf 52} (1950), 642-648


\end{thebibliography}

\end{document}